%% file: Transitive_PH_diffeomorphism_with_neutral_center.tex
\documentclass[12pt]{article}
\usepackage{amsmath}
\usepackage{amssymb}
\usepackage{graphicx}
\usepackage{graphics}
\usepackage{amsthm}
\usepackage{amscd}
\usepackage{color}
\usepackage{amsfonts}
\usepackage{fancyhdr}
\usepackage{fourier}
\usepackage[colorlinks=true,backref=page]{hyperref} 
\newtheorem{theorem}{Theorem}[section]
\newtheorem{mainthm}{Theorem}
\newtheorem*{theorem*}{Theorem}
\newtheorem{corollary}[theorem]{Corollary}
\newtheorem{proposition}[theorem]{Proposition}
\newtheorem{lemma}[theorem]{Lemma}
\newtheorem{Remark}[theorem]{Remark}
\newtheorem*{definition*}{Definition}

\newtheorem{claim}[theorem]{Claim}
\newtheorem{definition}[theorem]{Definition}
\newtheorem*{Question*}{Question}

\def\T{\mathbb{T}}

\def\norm #1{\Vert \,#1\, \Vert\,}
\makeatletter

\newcommand{\Rmnum}[1]{\expandafter\@slowromancap\romannumeral #1@}
\makeatother

\def\ud{\mathrm{d}}
\def\diff {\operatorname{Diff}}
\def\dim{\operatorname{dim}}
\def\homeo{\operatorname{Homeo}}
\def\Orb{\operatorname{Orb}}

\def\Int{\operatorname{Int}}
\def\Id{\operatorname{Id}}

 \def\NN{{\mathbb N}}  
 \def\RR{{\mathbb R}}  \def\TT{{\mathbb T}}
   
 \def\ZZ{{\mathbb Z}}

\def\La{\Lambda}

\def\e{\varepsilon}

\def\cA{\mathcal{A}}    

\def\cB{\mathcal{B}}   \def\cN{\mathcal{N}}

\def\cF{\mathcal{F}}  \def\cL{\mathcal{L}}

\numberwithin{equation}{section}         

\begin{document}
\vspace{-2cm}

\title 
{Transitive partially hyperbolic diffeomorphisms with   one-dimensional neutral center }
\author{Christian Bonatti and Jinhua Zhang \footnote{J.Z was supported by  the 
ERC project 692925 \emph{NUHGD}. }}

\vspace{-2cm}

\maketitle
\begin{abstract}
In this paper, we study transitive partially hyperbolic diffeomorphisms with one-dimensional topologically neutral center,  meaning that the length of the iterate of small center segments remains small. Such systems are dynamically coherent. We show that there exists a continuous metric along the center foliation which is invariant under the dynamics. 

As an application, we classify the transitive  partially hyperbolic diffeomorphisms on $3$-manifolds with topologically neutral center. 

\smallskip
\hspace{-1cm}\mbox
\smallskip

 
\end{abstract}

\section{Introduction}
A $C^1$ diffeomorphism$f$ on a  closed manifold $M$ is \emph{partially hyperbolic} if there exists a $Df$-invariant splitting
$TM=E^{s} \oplus  E^{c}\oplus  E^{u}$ such that $E^{s}$ is uniformly contracting,  $E^{u}$  is uniformly expanding and $E^{c}$ has the intermediate behavior; to be precise, there exists an integer $N\in\NN$ such that for any $x\in M$
\begin{itemize}
	\item {\bf{Contraction and expansion}} $$\|Df^N|_{E^s(x)}\|<\frac 12 \textrm{ and }
	\|Df^{-N}|_{E^u(x)}\|<\frac 12;$$
	\item {\bf{Domination}} $$\|Df^N|_{E^s(x)}\|\cdot\|Df^{-N}|_{E^c(f^N(x))}\|<\frac 12 \textrm{ and }  \|Df^N|_{E^c(x)}\|\cdot\|Df^{-N}|_{E^u(f^N(x))}\|<\frac 12.$$
\end{itemize}
\begin{definition}
	For a  $C^1$ partially hyperbolic diffeomorphism $f$ on $M$, one says   that $f$ is \emph{neutral along center},
	if there exists $C>1$ such that
	$$\frac{1}{C}<\norm{Df^n|_{E^{c}(x)}}<C,\textrm{ for  any $x\in M$ and $n\in\mathbb{Z}$}. $$
	
	One says that $f$ is \emph{topologically neutral along center}  if for any $\e>0$ there is $\delta>0$ so that  any $C^1$-center-path $\sigma$ of length bounded by $\delta$ has all its images $f^n(\sigma), n\in \ZZ$ bounded in length  by $\e$. 
\end{definition}

One easily checks  that if $f$ is neutral, then $f$ is topologically neutral.  However the reverse is not true: there are   partially hyperbolic diffeomorphisms on $3$-manifolds, with $1$-dimensional center bundle, which are topologically neutral but not neutral (see Section~\ref{s.dynamical coherent}).

For partially hyperbolic diffeomorphisms with neutral or topologically neutral center, the center bundle is  uniquely integrable due to ~\cite{HHU2}.

A \emph{center arc} is an equivalence class of locally injective center paths, up to changing the parametrization. A point is a degenerate arc.  

\begin{definition}
	We will call \emph{center metric} a function 
	$\ell^c$ defined on the set of arcs, with the following properties: 
	\begin{itemize}
		\item (positivity) strictly positive on the non-degenerate arcs, and vanishing on degenerate arcs.
		\item (additivity) consider $\sigma \colon [a,b]\to M$ a center path and $c\in [a,b]$ then 
		$$\ell^c(\sigma_{[a,c]}) +\ell^c(\sigma_{[c,b]})= \ell^c(\sigma_{[a,b]}).$$
		
		\item (continuity) if $\sigma_t$ are center arcs associated to a $C^0$-continuous family of center-paths, then $\ell^c(\sigma_t)$ varies continuously with $t$.  
	\end{itemize}
\end{definition}

\subsection{Results in any dimension}
Recall that a diffeomorphism on a connected closed  manifold $M$ is \emph{transitive} if it admits a dense orbit. In this paper, we work in $C^1$-scenario.

\begin{mainthm}~\label{thm.main-metric}
	Let $f$ be a $C^1$-partially hyperbolic diffeomorphism on a closed manifold $M$. Assume that $f$ has one-dimensional topologically neutral center and $f$ is transitive, then there exists a center metric which is invariant under $f$ (in other words, the action of $f$ on  center leaves is by isometries for this center-metric). 
	
	As a consequence, this center metric is invariant by the strong stable and strong unstable holonomies. 
	
	Furthermore the invariant center metric is unique up to multiplying by a (positive) constant.
\end{mainthm}

When the center bundle is orientable and $f$ preserves the orientation of the center, the center metric gives an continuous flow, by following the center leaves at constant speed. The invariance of  the center metric implies that  the constant speed flow is invariant  under the dynamics.
Thus next result is a straightforward corollary of Theorem~\ref{thm.main-metric}:

\begin{mainthm}~\label{thm.main-flow}
	Let $f$ be a $C^1$ partially hyperbolic diffeomorphism on a closed manifold $M$. Assume that
	\begin{itemize}
		\item[--] $f$ has one-dimensional topologically neutral center and $f$ is transitive;
		\item[--] $E^c$ is orientable and $f$ preserves its orientation;
	\end{itemize}
	then there exists a continuous flow $\{\varphi_t\}_{t\in\mathbb{R}}$ on $M$ with the following properties:
	\begin{itemize}
		\item $\{\varphi_t(x)\}_{t\in\mathbb{R}}=\cF^c(x)$ for any $x\in M$; in particular, $\{\varphi_t\}_{t\in\mathbb{R}}$ has no singularities;
		\item $f$ commutes with the flow $\varphi_t$, that is, $f\circ\varphi_t=\varphi_t\circ f$ for any $t\in\mathbb{R}$.
	\end{itemize}
\end{mainthm}


The following result gives the transitivity of a   partially hyperbolic diffeomorphisms with topologically neutral center provided that the orbit of some point is dense in an open set.  Under the setting of partial hyperbolicity and allowing an $\omega$-limit set to contain an open set, the usual way to recover transitivity  is to assume \emph{accessibility}. Here, we strongly use the topologically neutral property. 
\begin{proposition}\label{p.omega}
	Let $f$ be a $C^1$ partially hyperbolic diffeomorphism on a closed connected manifold $M$. Assume that
	\begin{itemize}
		\item $f$ has topologically  neutral center;
		\item there is $y\in M$ whose  $\omega$-limit
		set $\omega(y)$ has non-empty interior.
	\end{itemize} 
	Then $f$ is transitive. 
\end{proposition}

As a consequence, one has the following observation which has its own interest and  is  useful when the center bundle $E^c$ is not orientable, or $f$ does not preserve an orientation of it.

\begin{proposition}\label{p.lift}
	Let $f$ be a $C^1$ partially hyperbolic diffeomorphism on a closed manifold $M$. Assume that $f$ has topologically  neutral center and $f$ is transitive. Let $\pi:\widehat{M}\rightarrow M$ be a (connected) finite cover of $M$ and $\hat{f}$ be a lift of $f$ to $\widehat{M}$, and $k>0$ be an integer.  Then $\hat{f}^k$ is transitive. 
\end{proposition}
We remark that in Propositions~\ref{p.lift} and~\ref{p.omega}, we don't assume the center to be $1$-dimensional.

Considering non-transitive  partially hyperbolic diffeomorphisms with topologically neutral center,   we get the following result which may be useful for further studies: 

\begin{proposition}~\label{p.recurrent-center-saturated} Let $f$ be a $C^1$ partially hyperbolic diffeomorphism with $1$-dimensional topologically neutral center. Then the set of recurrent (resp. positively recurrent) points is saturated 
	by the center leaves. 
\end{proposition}

Let us finish these general results by observing that Theorems~\ref{thm.main-metric} and ~\ref{thm.main-flow} are no more true if one removes the transitivity hypothesis: \emph{consider the partially hyperbolic diffeomorphism $f$ built in ~\cite{BPP} which is non-transitive and has one dimensional neutral center; the example is obtained by composing a Dehn twist to the time $N$-map of a non-transitive Anosov flow which admits only one transitive attractor,  one transitive repeller and two transverse tori $T_1, T_2$ in the wandering domain; one can assume that the  Dehn twist is supported on an orbit segment of $T_1$;
	the dynamics of $f$ coincides with the time $N$-map of the Anosov flow, hence one has no-choice of the center metric on the repeller and the attractor since the dynamics in the orbit of $T_2$ coincides with the time $N$-map of the Anosov flow; however,  one can do a small perturbation in the support of the Dehn twist and one gets a new partially hyperbolic diffeomorphism with neutral center and does not admit invariant metric. As it is not main aim of this paper, we will not present all the details.  }

\subsection{Classification result in dimension $3$}
Given two diffeomorphisms $f,g$ on a closed manifold $M$, one says that $f$ is \emph{$C^0$-conjugate} to $g$ if there exists a homeomorphism $h$ on $M$ such that $h\circ f=g\circ h.$

Using Theorem~\ref{thm.main-metric}, we obtain the following \emph{classification up to conjugacy}:
\begin{mainthm}~\label{thm.main}
	Let $f$ be a $C^1$ partially hyperbolic diffeomorphism on a closed 3-manifold $M$. Assume that $f$ has one-dimensional topologically neutral center and $f$ is transitive, then up to  finite lifts and iterates,
	$f$ is $C^0$-conjugate to one of the followings:
	\begin{itemize}
		\item skew products over  a linear Anosov on $\mathbb{T}^2$ with the rotations of the circle;
		\item the time 1-map of a  transitive topological Anosov flow.
	\end{itemize}
\end{mainthm}

\begin{Remark}
	\begin{itemize}
		\item The example in ~\cite{BPP} (see also ~\cite{BZ}) shows that the transitivity assumption is necessary: there are   partially hyperbolic diffeomorphisms $f$ on $3$ manifolds with neutral center and admitting non-compact center leaves which are not periodic. Thus $f$ is not conjugated, and not even center-leaf conjugated, to any of the models in Theorem~\ref{thm.main}. 
		\item During the final preparation of this paper, we notice a paper by P. Carrasco, E. Pujals and  F. Rodriguez-Hertz~\cite{CPH}  proving a classification result under certain smooth rigid conditions. They work in $C^\infty$-setting and obtain a $C^\infty$-conjugacy result. Also, their techniques are different from the ones in this paper.
	\end{itemize}
	
\end{Remark}

As a consequence, one immediately gets the following
\begin{corollary}
	Let $f$ be a transitive partially hyperbolic diffeomorphism on a 3-manifold $M$. Assume that $f$ has one-dimensional topologically neutral center, then $f$ has compact center leaves. Furthermore, if there exist compact center leaves which are non-periodic, then the center foliation is uniformly compact. 
\end{corollary}

Our  result is motivated by the following question raised in ~\cite{H}.
\begin{Question*}
	Does there exist  a partially hyperbolic diffeomorphism with isometric action on the center bundle which is robustly transitive?
\end{Question*}
The evidence in~\cite{BG,S}  indicates the answer might be negative, but the question remains open.

\medskip

Let us briefly recall some historical background of this paper.  In a talk in 2001, E. Pujals  informally conjectured that the family of  transitive partially hyperbolic diffeomorphisms, up to isotopy,  falls into three parts: time one-map of a transitive Anosov flow, linear Anosov on $\mathbb{T}^3$ and skew-products over  linear Anosov maps on $\mathbb{T}^2$ with  rotations on the circle. Then observed by Bonatti-Wilkinson~\cite{BW}, one has  to take finite lifts and  iterates into account. Inspired by Pujals's conjecture, F. Rodriguez Hertz, J. Rodriguez Hertz and R. Ures conjectured that the family of  dynamically coherent partially hyperbolic diffeomorphisms, up to finite lifts and iterates, falls into three parts as in the conjecture of Pujals. Some partial results towards to these two conjectures have been obtained in ~\cite{BW,HaPo0, HaPo,Bo, Ca, Gogolev}. Then some counter-examples are constructed in~\cite{BPP,BGP,BGHP}. In ~\cite{BPP}, the authors built a dynamically coherent partially hyperbolic diffeomorphism on a 3-manifold which supports an Anosov flow, and the diffeomorphism neither has periodic  center foliation nor is isotopic to identity (therefore is a counter-example to Rodriguez Hertz-Rodriguez Hertz-Ures conjecture, and some generalization is obtained in~\cite{BZ}), furthermore, the example in ~\cite{BPP} is not transitive. In~\cite{BGP,BGHP}, the authors built robustly transitive partially hyperbolic diffeomorphisms on 3-manifolds which do not satisfy Pujals's conjecture, and the examples in~\cite{BGHP}  are designed to be non-dynamically coherent, but the dynamical coherence of examples in~\cite{BGP} is still unknown.

\medskip

\noindent{\it Acknowledgments.}
J. Zhang would like to thank  Institut de Math\'ematiques de Bourgogne for hospitality.

\section{Preliminary}
In this section, we   collect the notions and the known results used in this paper.
\subsection{Dynamical coherence}~\label{s.dynamical coherent}
Given a partially hyperbolic diffeomorphism $f$, one says that $f$ is \emph{dynamically coherent}, if there exist invariant foliations $\cF^{cs}$ and $\cF^{cu}$ tangent to $E^s\oplus E^c$ and $E^c\oplus E^u$ respectively.
When $f$ is dynamically coherent, it naturally induces the center foliation by taking the intersection of $\cF^{cs}$ and $\cF^{cu}$.

For partially hyperbolic diffeomorphisms, the strong stable and strong unstable bundles are always integrable, and they are integrated into  unique $f$-invariant foliations which will be called strong foliations, see ~\cite{HPS}. For the center bundle, the situation is more delicate; even in one-dimensional center case, there might not exist center foliations, see the examples in ~\cite{HHU4} and ~\cite{BGHP}.

Recall that $f$ has topologically neutral center if   for any $\e>0$, there exists $\delta>0$ such that for any $C^1$-path $\gamma$ tangent to $E^c$ of length bounded by $\delta$,   the length of $f^n(\gamma)$ is bounded by $\e$ for any $n\in\mathbb{Z}$.  
\begin{theorem}[Theorem 7.5 in~\cite{HHU2}]~\label{thm.dynamical-coherence}
	Let $f$ be a $C^1$ partially hyperbolic diffeomorphism. Assume that $f$ has topologically neutral center, then $f$ is dynamically coherent. Furthermore, the center bundle is uniquely integrable.
\end{theorem}
\begin{Remark}
	It is worth to notice that in Theorem 7.5~\cite{HHU2}, the \emph{plaque expansiveness} is also obtained (in this paper, we will not use this fact).
\end{Remark}

To end this subsection,   we  show that there exists  a transitive partially hyperbolic diffeomorphism whose center is  topologically  neutral but  not neutral. 

\begin{proposition}~\label{p.non-neutral-example}
	There exists a transitive partially hyperbolic diffeomorphism on $\TT^3$ with one dimensional topologically neutral center but not neutral.  
\end{proposition}

\proof[Proof of Proposition~\ref{p.non-neutral-example}]
Let  $R_{\alpha}$ be an irrational rotation on $S^1=\mathbb{R}/\mathbb{Z}$. As $R_\alpha$ has no periodic points, one can apply Theorem B in ~\cite{BCW} to get a $C^1$-diffeomorphism $h\in\diff^1(S^1)$ which is $C^1$-close enough to $R_\alpha$ such that
\begin{itemize}
	\item $$\lim_{n\rightarrow\infty}\inf_{x\in S^1}\big\{\|Dh^n(x)\|,\|Dh^{-n}(x)\| \big\}=\infty.$$
	\item $h$ is $C^0$-conjugate to $R_\alpha$.
\end{itemize} 
Let $A$ be a linear Anosov map on $\TT^2$,
then    $F:=A\times h$ is a partially hyperbolic diffeomorphism on $\TT^3$ with one-dimensional center. As $h$ is conjugate to a rotation, $F$ has topologically neutral but not neutral center bundle.  As $h$ is transitive and $A$ is topologically mixing (that is, for any open sets $U,V\subset\TT^2$, there exists  $N\in\mathbb{N}$ such that $A^n(U)\cap V\neq\emptyset$ for $n\geq N$), $F$ is transitive. 
\endproof

\subsection{Invariant foliations for partially hyperbolic diffeomorphisms with topologically neutral center}~\label{s.neutral center}

Let $f$ be a partially hyperbolic and dynamically coherent diffeomorphism, then one   has 
$$\cF^{ss}(\cF^c(x)):=\cup_{y\in\cF^c(x)}\cF^{ss}(y)\subset\cF^{cs}(x);$$ one says that the center stable foliation is \emph{complete} if  
$$\cF^{cs}(x)=\cF^{ss}(\cF^c(x)) \textrm{ for any $x\in M$}.$$

To our knowledge, it is still open if the center stable foliation is complete for all dynamically coherent partially hyperbolic diffeomorphisms. For the case where $f$ is partially hyperbolic diffeomorphism with one dimensional neutral center, it has been proved in ~\cite{Z} that  its invariant foliations are complete and the topology of the center stable leaves is described.
\begin{theorem}[Theorem A in \cite{Z}]~\label{thm.leaf-structure} Let $f$ be a partially hyperbolic diffeomorphism on a 3-manifold  with one dimensional  neutral center. Then one has that
	\begin{itemize}
		\item the center stable and center unstable foliations are complete;
		\item each leaf of center stable (resp. center unstable ) foliation is a plane, a cylinder or a M$\ddot{o}$bius band;
		\item a center stable (resp. center unstable) leaf is a cylinder or a M$\ddot{o}$bius band if and only if such center stable (resp. center unstable) leaf contains a compact center leaf.
	\end{itemize}
\end{theorem}
\begin{Remark}
	Indeed, the second and the third items are the consequences of completeness of center stable (resp. center unstable) foliations. 
\end{Remark}

According to  Theorem~\ref{thm.leaf-structure},  in the case where there is no compact center leaves, the center stable and center unstable leaves are planes, and in this case one can know which manifold supports such partially hyperbolic diffeomorphism by the following result:
\begin{theorem}[Theorem 3 in \cite{R}]~\label{thm.rosenberg}
	Let $M$ be a closed 3-manifold. Assume that there exists a $C^0$-foliation on $M$ whose leaves are all planes, then $M$ is $\mathbb{T}^3.$
\end{theorem}
\begin{Remark} Theorem~\ref{thm.rosenberg} is first proved by H. Rosenberg~\cite{R} assuming that the foliation is $C^2$. Then observed by D. Gabai, the result holds for $C^0$-foliation due to ~\cite[Theorem 3.1]{I}, and the proof can be found  in \cite[Section 3]{G}.
\end{Remark}

The completeness of center stable and center unstable foliations can also be obtained in the topologically neutral case.
\begin{proposition}~\label{p.complete}
	Let $f$ be a $C^1$ partially hyperbolic diffeomorphism with topologically neutral center, then the center stable and center unstable foliations are complete. 
\end{proposition}
The proof follows as the one of Theorem A in ~\cite{Z}. Here, we sketch the proof. 
\proof[Sketch of the proof]
By Theorem~\ref{thm.dynamical-coherence},  $f$ is dynamically coherent and the center is uniquely integrable. Furthermore, there exist $\delta_1>0$ and $\delta_2>0$ such that for any $x$,  if  $y\in\cF_{\delta_1}^{cs}(x)$, then $\cF^{ss}_{\delta_2}(y)$ intersects $\cF^c_{\delta_2}(x)$ into a unique point.  

If the center stable leaf is not complete, then there exists a point $x_0$ such that $\cF^{ss}(\cF^c(x_0))\subsetneq  \cF^{cs}(x_0)$. In this case, by Proposition 1.3 in ~\cite{BW}, there exists a strong stable leaf $\cF^{ss}(y_0)\subset\cF^{cs}(x_0)$ such that 
\begin{itemize}
	\item $\cF^{ss}(y_0)$ is disjoint from  $\cF^{ss}(\cF^c(x_0))$;
	\item there exists an arbitrarily short  center path $\sigma$ whose two endpoints are in $\cF^{ss}(y_0)$ and $\cF^{ss}(\cF^c(x_0))$ respectively. 
\end{itemize}
By iterating $\sigma$ forwardly, for $n$ large enough $f^n(\sigma)$ has one endpoint close enough to a point in $\cF^c(f^n(x_0))$ and the other endpoint   in $\cF^{ss}(f^n(y_0))$ which is uniformly away from $\cF^c(f^n(x_0))$ in this case, the length of $f^n(\sigma)$ can be arbitrarily large contradicting to the topologically  neutral property. 
\endproof 
\subsection{Previous classification results on $3$-manifolds}
In this section, we recall some  classification results of partially hyperbolic diffeomorphisms on $3$-manifolds which are used in this paper (we refer the readers to a survey ~\cite{HaPo1}  and references therein for more results on classification).

In ~\cite{BW}, the authors  classified certain transitive partially hyperbolic diffeomorphisms on $3$-manifolds. As we are in the setting of dynamical coherence, for simplicity, we will present  a weaker version of  Theorems 1 and 2 in~\cite{BW}.

\begin{theorem}~\label{thm.bonatti-wilkinson}
	Let $f$ be a $C^1$ partially hyperbolic diffeomorphism on a closed $3$-manifold $M$. Assume that $f$ is transitive and  dynamically coherent.
	\begin{itemize}
		\item If there exists a compact and invariant center leaf $L$ such that $W_{\delta}^s(L)\cap W_{\delta}^u(L)\setminus\{L\}$ contains a compact center leaf for some   $\delta>0$, then up to finite lifts, $f$ is $C^0$-conjugate to a skew-product;
		\item If there exists a compact and periodic center leaf $L$ such that every center leaf in $W^s(L)$ is   periodic under $f$, then there exist $n\in \NN$ and $c>0$ such that \begin{itemize}
			\item every center leaf is $f^n$-invariant;
			\item   for any  $x\in M$, the distance of $x$ and $f^n(x)$ on the center leaf is bounded by $c$;
			\item  the center foliation carries  a continuous flow $C^0$-conjugate to an expansive transitive flow.
		\end{itemize}
	\end{itemize}    
\end{theorem}

Given  two partially hyperbolic and dynamically coherent diffeomorphisms  $f$ and $g$ on $M$ , one says that $f$ is \emph{leaf conjugate} to $g$,   if    there exists a homeomorphism $h\in \homeo(M)$ such that
for any $x\in M$
\begin{itemize}
	\item $h(\cF_g^c(x))=\cF^{c}_f(h(x))$;
	\item   $h\circ g(\cF_g^c(x))=f\circ h(\cF_g^c(x))$.
\end{itemize}

Each  $f\in\diff^1(\mathbb{T}^3)$ induces an action on the fundamental group of $\TT^3$: $f_*:\pi_1(\TT^3)\mapsto\pi_1(\TT^3)$ which is called the \emph{linear part of $f$}. 
\begin{theorem}[Theorem 1.3 in \cite{HaPo0}]~\label{thm.hammerlindl-potrie}
	Let $f$ be a dynamically coherent partially hyperbolic diffeomorphism on $\TT^3$, then $f$ is leaf conjugate to its linear part $f_*$.
\end{theorem}
As a consequence, one has the following result. 
\begin{proposition}~\label{p.existence-of-compact-center-leaf}
	Let $f$ be a partially hyperbolic diffeomorphism on a closed 3-manifold $M$. Assume that $f$ has $1$-dimensional topologically neutral center, then $f$  has compact center leaves. 
\end{proposition}
\proof
Theorem~\ref{thm.dynamical-coherence} gives the dynamical coherence of $f$. 
Assume, on the contrary, that $f$ does not admit any compact center leaves. Then by Theorem~\ref{thm.leaf-structure}, all the center stable leaves are planes. By Theorem~\ref{thm.rosenberg}, one has that $M=\mathbb{T}^3$. By Theorem~\ref{thm.hammerlindl-potrie},  $f$ is leaf conjugate to its linear part $f_*:\pi_1(\TT^3)\mapsto\pi_1(\TT^3)$. Since the center stable leaves are planes, $f$ is isotopic to a linear Anosov map $A=f_*$ on $\mathbb{T}^3$, therefore, $f$ is semi-conjugate to $A$ (for a proof see for instance ~\cite{P}). Moreover, the semi-conjugacy sends the center leaves of $f$ to the center leaves of $A$, and on each leaf the semi-conjugacy  maps at most  countably many center segments of $f$ into points (see ~\cite{U}). Let $p$ be a fixed point of $f$, then $f|_{\cF^c(p)}:\cF^c(p)\mapsto\cF^c(p)$ is semi-conjugate to a contracting or expanding affine map  on $\RR$, which  contradicts to the neutral property on the center. 
\endproof

\subsection{ H\"older Theorem}
In this part, we recall the H\"older Theorem for actions on one dimensional manifolds. The action given by a group $\Gamma$ acting on a manifold $M$ is  a  \emph{free action} if each  non-trivial element in $\Gamma$ has no fixed points.
\begin{theorem}~\label{thm.holder}
	Let $\Gamma$ be a   group of orientation preserving homeomorphism acting freely on $\RR$ (resp. $S^1$). Then $\Gamma$ is isomorphic to a subgroup of translations on $\RR$ (resp. a subgroup of $SO(2)$).
\end{theorem}
The proof of Theorem~\ref{thm.holder}  can be founded in ~\cite{Na} (see Propositions 2.2.28 and 2.2.29,  and Theorem 2.2.23 therein).

\section{$\omega$-limit sets with non-empty interior}

The aim of this section is to prove  Proposition~\ref{p.omega}. 

\begin{lemma}~\label{l.strong-foliation-saturated} 
	Let 	$f$  be a $C^1$ partially hyperbolic diffeomorphism. Assume that there is a point $y\in M$  whose $\omega$-limit set $\omega(y)$   has non empty interior. 
	Then  $\omega(y)$ is saturated by 
	strong stable and strong unstable leaves. 
	
	Furthermore, if $f$  has  topologically neutral center bundle, then $\omega(y)$ is also saturated by center leaves.
\end{lemma}
\begin{proof}
	Notice that the interior $\Int(\omega(y))$ of $\omega(y)$ is  $f$-invariant. As the positive orbit of $y$ meets the interior of $\omega(y)$, one has $y\in\Int(\omega(y))$. 
	Thus the restriction of $f$ to $\omega(y)$ is a transitive homeomorphism, and therefore there is $x\in \Int(\omega(y))$ so that $\alpha(x)=\omega(x)=\omega(y)$. Since the orbit of $x$ is dense  in $\omega(y)$, it suffices to show that $\cF^{ss}(x),\cF^{uu}(x)$ are contained $\omega(y)$.

	Since $x$ is in interior of $\omega(y)$, there exists $\delta_0>0$ such that the $\delta_0$-neighborhood of $x$ in $M$ is contained in $\Int(\omega(y))$. For any point $z\in\cF^{ss}(x)$, there exists  $n_z\in\NN$ such that $\ud(f^n(x), f^n(z))<\delta_0/2$ for any $n\geq n_z$. As $x$ is recurrent, there exists an integer $N>n_z$ such that $f^N(x)$ is in the $\delta_0/3$-neighborhood of $x$. Therefore $f^N(z)\in \Int(\omega(y))$. By the $f$-invariance of $\Int(\omega(y))$, one has $z\in\Int(\omega(y))$. 	By the arbitrariness of $z$, one has   $\cF^{ss}(x)\subset \omega(y)$. 
	Analogously, one can show that $\cF^{uu}(x)\subset \omega(y)$.

	Now, we prove the `furthermore' part. Since  the $\delta_0$-neighborhood of $x$ is contained in the interior of $ \omega(y)$, one has that  $\cF_{\delta_0/2}^c(x)\subset \omega(y)$. As the forward orbit of $x$ is dense in $ \omega(y)$, by the topologically neutral property, there exists $\eta_0>0$ such that for any point $z\in  \omega(y)$, one has that $\cF^c_{\eta_0}(z)$ is contained in $\omega(y)$. By the arbitrariness of $z$, one has that  $\cF^c(x)\subset  \omega(y)$ which by the density of the orbit of $x$ implies that $\omega(y)$ is saturated by center leaves.  
\end{proof}

\begin{proof}[Ending the proof of Proposition~\ref{p.omega}]
By 	Lemma~\ref{l.strong-foliation-saturated}, the set  $\omega(y)$ is saturated by strong foliations and center foliation.
	Any set  which is saturated by the $3$ foliations $\cF^{ss}$, $\cF^{uu}$ and $\cF^c$ is open. 
	As $\omega(y)$ is compact, one gets $\omega(y)=M$ as $M$ is assumed to be connected, concluding. 
\end{proof}

\input{Center-flow}

\input{Existence-periodic-compact}

\input{First-intersection}
\input{Shadowing}

\input{Biblio}


 \vskip 5pt
 
\noindent Christian Bonatti,

\noindent {\small Institut de Math\'ematiques de Bourgogne\\
UMR 5584 du CNRS}

\noindent {\small Universit\'e de Bourgogne, 21004 Dijon, FRANCE}

\noindent {\footnotesize{E-mail : bonatti@u-bourgogne.fr}}

\vskip 2mm
\noindent Jinhua Zhang,

\noindent{\small Laboratoire de Math\'ematiques d'Orsay,\\
UMR 8628 du CNRS,\\
Universit\'e Paris-Sud,
 91405 Orsay, France.}
\\
\noindent {\footnotesize{E-mail :  zjh200889@gmail.com}}\\

\end{document}

%% file: Center-flow.tex
\section{Existence of invariant center metric : Proof of Theorem~\ref{thm.main-metric}}~\label{s.center-flow}

Throughout this section, we   assume that $f$ is a $C^1$ partially hyperbolic diffeomorphism on a closed connected manifold $M$ with one dimensional topologically neutral center. By Theorem~\ref{thm.dynamical-coherence}, $f$ is dynamically coherent and the center bundle is uniquely integrable.

The  aim of this section is to show that if $f$ is transitive, one can define a center metric which is invariant under $f$. In this section, for notational convenience, we use $L$ or $L_i$ to denote a center leaf.
\subsection{Limit center maps}
\begin{definition}~\label{def:limit-center-map} Let $L_1$ and $L_2$ be center leaves of $f$.  Consider a map $F\colon L_1\to L_2$.  We say that $F$ is a \emph{limit center map} if there is a sequence $n_i\in\NN$ with $|n_i|\to \infty$ so that the sequence $f^{n_i}|_{L_1}$ pointwise converges to $F$. 
\end{definition}
\begin{Remark}~\label{r.uniform-convergence-of-center-limit}
	By   continuity of the center foliation and the topologically neutral property, in Definition~\ref{def:limit-center-map}, for each compact center path $\sigma\subset L_1$, the convergence of $f^{n_i}|_{\sigma}$ is uniform.
\end{Remark}

The next result gives the existence of limit center maps between certain center leaves.
\begin{lemma}\label{l.exists-limit-maps}
	Let $f$ be a partially hyperbolic diffeomorphism on $M$ with $1$-dimensional, topologically neutral center bundle.
	For any $x,y\in M$, if   $y\in\omega(x)$, then $\cL(L_x, L_y)\neq\emptyset.$
	
	More precisely, if the sequence $f^{n_t}(x)$ converges to $y$ then one can extract a subsequence  $\{m_j\}$ of the sequence $\{n_t\}$ so that the restriction $f^{m_j}|_{L_x}$ converges to a limit center map $F: L_x\to L_y$ with $F(x)=y$. 
\end{lemma}
\begin{proof}
	Let $\{x_i\}_{i\in\NN}$ be a dense subset of $L_x$. Assume that for an integer $j\in\NN$,  one has subsequences $\{n_k^j\}\subset \{n_k^{j-1}\}$ of $\{n_t\}$  such that $f^{n_k^j}(x_l)$ converges to a point on $L_y$ when $k$ tends to infinity for each $l\leq j$. Now, by the topologically  neutral property along the one-dimensional center bundle and $f^{n_t}(x)$ tending to $y$, there exists a  subsequence $\{n_k^{j+1}\}\subset\{n_k^{j} \}$ such that $f^{n_k^{j+1}}(x_{j+1})$ converges to a point on $L_y$. Then the diagonal argument provides a subsequence $\{m_j\}$ of $\{n_t\}$ such that for each $l$, $f^{m_j}(x_l)$ converges to a point on $L_y$ when $j$ tends to infinity. The topologically  neutral property and the continuity of  center foliation give that $f^{m_j}|_{L_x}$ pointwise converges to a limit center map.
\end{proof}

Now, we give some basic properties of limit center maps.
\begin{lemma}\label{l.limit-maps} Let $f$ be a partially hyperbolic diffeomorphism. Assume that the center bundle is  $1$-dimensional and  topologically neutral. 
	Then one has 
 \begin{enumerate}
  \item The set of   limit center maps are uniformly topologically neutral in the following sense: for any $\e>0$ small, there exist $\delta>0$ and $\eta>0$ such that for any limit center map $F: L_1\to L_2$, and two points $x,y\in L_1$, one has 
  \begin{itemize}
  	\item[--] If $\ud^c(x,y)<\delta$, then $\ud^c(F(x), F(y))<\e,$  where $\ud^c(\cdot,\cdot)$ denotes the distance on center leaves; in particular, $F$ is continuous;
  	\item[--] If $\e_0/4>\ud^c(x,y)>\e$, then $\ud^c(F(x), F(y))>\eta,$ where $\e_0$ is a lower bound for the length of center leaves. 
  \end{itemize}
\item Each   limit center map $F$ from $L_1$ to $L_2$  is a local homeomorphism and is surjective;
  \item If $F: L_1\to L_2$ and $G: L_2\to L_3$ are limit center maps, then $G\circ F$ is a limit center map from $L_1$ to $L_3$.
  \item  If  $F\colon L \to L$ is a limit center map having a fixed point $x\in L$, then  
  \begin{itemize}
   \item[--] $F$ is the identity map of $L$ provided that  $F$ is orientation preserving;
   \item[--] $F$ is an involution on $L$ (i.e.   $F^2=\Id_L$) provided that  $F$ is orientation reversing. 
  \end{itemize}
 
 \item If   $F\colon L\to L$ is a limit center map, then $F$ is a homeomorphism. 
 \end{enumerate}
\end{lemma}

\begin{proof}
By   topologically neutral property, for any $\e>0$, there exists $\delta>0$ such that any center path $\sigma$ of length bounded by $\delta$ has its images $f^i(\sigma)$ whose length is bounded by $\e$ for any $i\in\ZZ;$ by the continuity of center foliation, this in particular gives that for any limit center map $F:L_1\to L_2$ and any  two points $x,y\in L_1$ with $\ud^c(x,y)<\delta$,   one has $\ud^c(F(x),F(y))<\e.$ On the other hand, if there exists $\e_1>0$ such that for any $n\in\NN$, there exists a limit center map $F_n: L_1^n\to L_2^n$ and two points  $x_n,y_n\in L_1^n$ such that  $\ud^c(x_n,y_n)>\e_1$ and $\ud^c(F_n(x_n), F_n(y_n))<\frac 1n$, that is, there exists   center paths  whose length is uniformly bounded from below and some of whose images have length arbitrarily small,  which contradicts to the topologically neutral property.  This proves the first item. 

By the definition of limit center  maps and the continuity of center foliation, each limit center  map is surjective. 
Since the center bundle is non-degenerate everywhere, there exists $\e_0>0$ such that the length of each  compact center leaf is bounded from below by $\e_0.$  Then there exists $\delta_0\in(0,\delta_0/2)$ such that for any two points $x,y$ in a same center leaf with $\ud^c(x,y)<\delta_0$, one has $\ud^c(f^i(x),f^i(y))<\e_0/4$ for $i\in\ZZ.$  Thus, for any limit center map $F:L_1\to L_2$ and  any   center path $\sigma\subset L_1$ of length $\delta_0$, the length of $F(\sigma)$ is bounded by $\e_0/4$;  by  the topologically neutral property of $F$, one has that $F:\sigma\to F(\sigma)$ is  injective and therefore is a homeomorphism. This proves the second item.  

Given two limit center  maps $F:L_1\to L_2$ and $G: L_2\to L_3$, by the second item, the map $G\circ F: L_1\to L_3$ is a local homeomorphism. Let $\{n_i\}$ and $\{m_i\}$ be the sequence of integers such that $f^{n_i}|_{L_1}$ and $f^{m_j}|_{L_2}$ converge to $F$ and $G$  respectively. Let $\{x_i\}_{i\in\NN} $ be a dense subset of $L_1$. 
Let $\{\e_n\}_{n\in\NN}$ be a sequence of positive numbers such that $\e_n$ tends to $0$. For $\e_1$ and $x_1$, by the choices of $\{n_i\}$ and $\{m_i\}$, there exist $l_1\in\{n_i\}$ and $k_1\in\{m_i\}$ such that $f^{k_1+l_1}(x_1)$ is $\e_1$ close to $G\circ F(x_1)$. Assume that one already has $|l_1|<|l_2|<\cdots<|l_{i}|$ which are in $\{n_j\}$ and $|k_1|<|k_2|<\cdots <|k_i|$ which are in $\{m_j\}$ such that   $f^{l_t+k_t}(x_j)$ is $\e_t$ close to $F\circ G(x_j)$ for any $j\leq t\leq i$. Once again, by the choice of $\{n_j\}$ and $\{m_j\}$, for $\e_{i+1}$ and $x_1,\cdots,x_{i+1}$, there exist $|l_{i+1}|>|l_i|$ and $|k_{i+1}|>|k_i|$ such that $f^{k_{i+1}+l_{i+1}}(x_j)$ is $\e_{i+1}$ close to $G\circ F(x_j)$ for $j\leq i+1$. Then one gets a sequence of integers $\tau_i=k_i+l_i$ such that for any $j$,  $f^{\tau_i}(x_j)$ tends to $G\circ F(x_j)$ when $i$ tends to infinity. The topologically  neutral property and continuity of center foliation gives that  $f^{\tau_i}|_{L_1}$ pointwise converges to $G\circ F.$ This gives the third item.

Let $F: L\to L$ be a limit center map  with   fixed points. If $F$ preserves the orientation, let $p$ be a fixed point and  $I\subset L$ be a small center segment  such that $F|_I: I\to F(I)$ is a homeomorphism and  $p$ is an endpoint of $I$. As $F$ is topologically neutral,   all the points on $I$ are fixed points of $F$. By the arbitrariness of $p$ and $I$, one has that $F: L\to L$ is $\Id_L.$ If $F$ reverses the orientation, then by the second item,  $F^2: L\to L$ is a limit center map with fixed points and preserving the orientation, therefore $F^2$ is $\Id_L.$

Let $F: L\to L$ be a limit center map. If $L$ is homeomorphic to $\RR$, as $F$ is a local homeomorphism and is surjective, $F$ is a homeomorphism. If $L$ is homeomorphic to $\mathbb{S}^1$, since the limit center map $F: L\to L$ is a local homeomorphism on $L,$ $F$ is an endomorphism on $L$ of degree $d\in\ZZ.$ If $|d|\neq 1,$ $F$ is a covering map and therefore $F$ has periodic points, thus there exists $n\in \NN$ such that  $F^n: L\to L$ has fixed points; by the forth item, $F^{2n} :L\to L$ is $\Id_L$ which gives the contradiction.
\end{proof}

Lemma~\ref{l.limit-maps} motivates the following notions : 

\begin{definition}Consider center leaves $L, L_1,L_2$.  We denote by 
	$\cL(L)$ (resp. $\cL(L_1,L_2)$)  the set of all limit center maps  from $L$ to $L$ (resp. from $L_1$ to $L_2$). 
	
	We denote by $\cL^+(L)$ the subset of orientation preserving limit center maps from $L$ to $L$. 
	
	We denote by $\cL^s(L_1,L_2)$ (resp. $\cL^u(L_1,L_2)$)  the subset of limit center maps from $L_1$ to $L_2$ obtained as limit of sequences $f^{n_i}|_{L_1}$ with $n_i\to +\infty$ (resp. $n_i\to -\infty$). 
	
	We define in the same way $\cL^s(L)$ and $\cL^u(L)$. 
\end{definition}

Now, we   give the proof of Proposition~\ref{p.recurrent-center-saturated}.
\proof[Proof of Proposition~\ref{p.recurrent-center-saturated}]
Let $x$ be a recurrent point. Then by Lemma~\ref{l.exists-limit-maps}, there exists a limit center map $F:L_x\to L_x$ having $x$ as a fixed point. By the third  and forth items in Lemma~\ref{l.limit-maps}, $\Id_{L_x}$ is a limit center map, therefore, every point on $L_x$ is recurrent. 
\endproof 

\begin{corollary}\label{c.strong-intersection}Let $L$ be a center leaf containing $x, y$ with $y\in\omega(x)$.  Then every 
strong stable leaf cuts $L$ in at most $1$ point. 
\end{corollary}
\begin{proof}
	Consider a sequence $m_j\to+\infty$, given by Lemma~\ref{l.exists-limit-maps},   so that $f^{m_j}$ converges to a limit center map $F:L\to L$.
	Let us argue by contradiction. Assume that $z_1 \neq  z_2$ are points in $L$ with  $z_2\in \cF^{ss}(z_1)$, thus  $F(z_1)=F(z_2)$ and  $F$ is not a homeomorphism, 
contradicting to the fifth  item of  Lemma~\ref{l.limit-maps}. 
\end{proof}

\subsection{Limit center maps for transitive diffeomorphisms}~\label{s.limit-center-map-for-transitive-diffeo}
Let $f$ be  a transitive partially hyperbolic diffeomorphism with $1$-dimensional topologically neutral center bundle.  
We denote $\cN=\{x\in M: \alpha(x)=\omega(x)=M \}$,  then $\cN$ is $f$-invariant. As $f$ is transitive, then $\cN$ is a residual subset of $M$. 

We will build metrics along  center leaves in the residual subset $\cN$ of $M$ and we will show that the metric we built is $f$-invariant, continuous and invariant under holonomies of strong stable and strong unstable foliations. 

 In our setting, we show that $\cN$ is saturated by center leaves and we give some description of the sets of limit center maps.

\begin{proposition}~\label{p.limit-group}
	Let $f$ be a $C^1$ partially hyperbolic diffeomorphism. Assume that $f$ is transitive and  has $1$-dimensional topologically neutral center.
	
	Then for any center leaf $L$ containing a point in $\cN$, one has 
	\begin{itemize}
		\item $L\subset\cN$; 
		\item If $L$ is not compact, then  there is a homeomorphism $\psi_L\colon L\to \RR$ so that $$\cL^+(L)=\big\{\psi_L^{-1}\circ T_t\circ \psi_L, t\in\RR\big\}$$ where $T_t$ is the translation
		$T_t\colon \RR\to \RR,  s\mapsto s+t$. 
	In this case,  either $\cL(L)=\cL^+(L)$ or $\cL(L)$ is the group of homeomorphisms generated by $\cL^+(L)$ and $\psi_L^{-1}\circ (-\Id_{\RR})\circ \psi_L$.
		
		Furthermore $\psi_L$ is unique up to composition by  an affine map of $\RR$.
		\item If $L$ is compact,  then there is a homeomorphisms $\psi_L\colon L\to S^1=\RR/\ZZ$ so that  $$\cL^+(L)=\big\{\psi_L^{-1}\circ R_t\circ \psi_L, t\in S^1\big\}$$ where  $R_t$ is the rotation
		$T_t\colon S^1\to S^1,  s\mapsto s+t ~(\textrm{mod $\ZZ$})$. 
	In this case,  either $\cL(L)=\cL^+(L)$ or $\cL(L)$ is the group of homeomorphisms generated by $\cL^+(L)$ and $\psi_L^{-1}\circ (-\Id_{S^1})\circ \psi_L$.
		
		Furthermore $\psi_L$ is unique up to composition by  a rotation of $S^1$.
	\end{itemize} 	
 
\end{proposition}
\begin{proof}
Let $\e_0>0$ be the infimum of the lengths of compact center leaves	if compact center leaves exist, otherwise one takes $\e_0=1.$

Fix a point $x\in L\cap \cN$. Since $L$ is one dimensional, one gives an orientation to it. 
For any $\e\in(0,\e_0/4)$, 
let $I^+_\e=[x,x^+_\e]^+$ (resp. $I_\e^-=[x^-_\e,x]^+$) be a center segment whose length is   $\e$ and the direction pointing from $x$ to $x_\e^+$ (resp. $x_\e^-$) through $I_\e^+$ (resp. $I^-_\e$) coincides with the positive (resp. negative)  direction of $L$.  
\begin{claim}~\label{c.small-translation}
	For  $\e\in(0,\e_0/4)$, there exists a limit center map $F\in\cL^+(L)$ (resp. $G\in\cL^+(L)$) sending $x$ to a point in $I^+_\e\setminus\{x\}$ (resp.  $I^-_\e\setminus\{x\}$). Moreover, such limit center maps can be obtained by  the forward and backward  iterates of $f$ respectively.
\end{claim}
\proof
We only deal with the case for $I^+_\e$and prove the claim only using the fact $\omega(x)=M$ (the other cases   follow analogously).

As $\omega(x)=M$, by Lemma~\ref{l.exists-limit-maps}, there exists a limit map $\hat{F}: L\to L$ sending $x$ to $x^+_\e$.
If $\hat F$ preserves the orientation, one can conclude. If $\hat F$ reserves the orientation, by the forth item of Lemma~\ref{l.limit-maps}, one has that $\hat{F}^2=\Id_L$. In this case $\hat F(x^+_\e)=x$, therefore there exists a $\hat F$-fixed  point $z_\e\in\Int(I^+_\e)$.  Now, consider a limit center map $\hat H: L\to L$ sending $x$ to $z_\e$. If $\hat H$ preserves the orientation, one can also conclude. 
If $\hat H$ reverses the orientation, we consider the map $F=\hat F\circ\hat{H}$ which is a limit center map from $L\to L$ by the third item of Lemma ~\ref{l.limit-maps} and preserves the orientation of $L$. Since $\hat H(x)=z_\e$ and $\hat{F}(z_\e)=z_\e$, one has $F(x)=z_\e$.
\endproof

 Now, we   show that there exist limit center maps preserving the orientation and sending $x$ to any point in $L.$ To be precise:
\begin{claim}~\label{c.all-translation}
For any point $y\in L$,   there exists a limit center map $F\in\cL^+(L)$ which sends $x$ to $y.$ Moreover, one can obtain such limit center maps by the forward  as well as the backward iterates of $f$.
\end{claim} 
\proof 

Consider the set 
$$\cA:=\big\{y\in L: \textrm{ there exists $F\in\cL^+(L)$ such that $F(x)=y$ } \big\}.$$ The claim is reformulated as $\cA=L.$ It suffices to show that one can obtain  limit center maps which preserve the orientation  send $x$ to any point in $L$ by the forward iterates of $f$. The other case would follow analogously.

By Claim~\ref{c.small-translation}, the set $\cA$ is non-empty. We will show that $\cA$ is a   closed subset of $L$.  Let $\{y_n\}_{n\in\NN}$ be a sequence of points in $\cA$ which tends to $y_0$ according  to the distance on $L$. Now, one fixes a small neighborhood of  $y_0$. Then one gives an orientation to  those center plaques in this neighborhood of $y_0$  according to the orientation of the  local center plaque $L_{loc}(y_0)$ of $y_0$. For any $l\in\NN$,  take $y_{n_l}$ which is $1/l$ close to $y_0$ on $L$. As $y_{n_l}\in \cA$,   one can choose $m_l\in\NN$ large enough such that $f^{m_l}(x)$ is $1/l$ close to $y_{n_l}$ and $f^{m_l}: {L_{loc}(x)}\to {L_{loc}(f^{m_l}(x))}$ preserves the orientation of local plaques. 
Now, one gets a sequence of positive integers $\{m_l\}$ tending to infinity such that $f^{m_l}(x)$ tends to $y_0$ and  $f^{m_l}: {L_{loc}(x)}\to {L_{loc}(f^{m_l}(x))}$ preserves the orientation of local plaques. By Lemma~\ref{l.exists-limit-maps}, there exists a limit center map $F\in\cL^+(L)$ with $F(x)=y_0$.

If $\cA\neq L$, then there exists an open center path $(a, b)\cap \cA=\emptyset$ and one endpoint is in $\cA$. Without loss of generality, one can assume $a\in \cA$. Let $F_a\in\cL^+(L)$ be a map such that $F_a(x)=a$. For $\e>0$,   consider the center path  $I^+_\e=[x, x_\e^+]^+$ as before. By Claim~\ref{c.small-translation}, there exists $F_\e\in\cL^+(L)$ sending $x$ to a point in the interior of $I^+_\e$. As the limit center maps are uniformly topological neutral (due to the first item of Lemma~\ref{l.limit-maps}), for $\e>0$, the limit center map $F_\e\circ F_a$ sends $x$ to a point in $(a,b)$ which gives the contradiction. 
\endproof

Now, consider the sets $$\cB^+=\big\{y\in L: \omega(y)=M \big\} \textrm{ and  }\cB^-=\big\{y\in L: \alpha(y)=M \big\},$$  which are non-empty since  $x\in\cB^+\cap \cB^-$. The following claim gives that $\cN$ is saturated by center leaves. 
\begin{claim}
	$\cB^+=\cB^-=L.$
\end{claim}
\proof 
One only needs to deal with  $\cB^+$ and the case for $\cB^-$ would follow analogously.

We will first show that $\cB^+$ is a closed subset of $L$. Let $z_n$ be a sequence of points in $L$ such that $z_n$ tends to a point $z\in L$ according to the distance on $L$. For $\e>0$, let $z_n$ be a point close enough to $z$ such that for shortest center path $\sigma_n$ connecting $z_n$ and $z$,   the length of $f^i(\sigma_n)$ is bounded by $\e/2$ for $i\in\ZZ$ due to the topologically neutral property on the center bundle.  As $z_n\in\cB$, one can take  $m_n\in\NN$ large such that $\ud(f^{m_n}(z_n),x)<\e/2,$
which implies that $\ud(f^{m_n}(z),x)<\e$. The arbitrariness of $\e$  gives $x\in\omega(z)$ which implies $z\in\cB^+.$

Assume, on the contrary, that $\cB^+$  is not the whole center leaf $L$.   Since $\cB^+$ is closed in $L$, there exists a center path  $\sigma=[z,w]\subset L$ such that
\begin{itemize}
	\item its interior is disjoint from $\cB^+$;
	\item one of its endpoint is in $\cB^+;$
		\item the orientation pointing from $z$ to $w$ in $\sigma$ coincides with the positive orientation of $L$.
\end{itemize}  
 Without loss of generality, one can assume that    $w\in \cB^+$. 
 By topologically neutral property on the center bundle, there exists $\delta_0>0$ such that  the length of $f^i(\sigma)$ is bounded from below by $\delta_0$ for any $i\in\ZZ$. Consider     a short center path  $[x,p]$ in $L$ such that its length is much smaller $\delta_0$ and the orientation of $[x,p]$ pointing from $x$ to $p$ coincides with the positive orientation of $L$.
As $w\in\cB^+$, one can apply Claim ~\ref{c.all-translation} to $w$ with respect to the forward iterates of $f$, and one gets a limit center map  $F:L\to L$ which is orientation preserving and maps $w$ to $p$. This implies that there exists a point $w_0$ in the interior of $[z,w]$ whose $\omega$-limit set contains $x$. As $x\in\cN$, one has $w\in\cB^+$ and one obtains the contradiction.  
\endproof 

In the following, we will show that $\cL^+(L)$ is a group; in particular, this implies  that the limit center map in $\cL^+(L)$ sending one specific point to another one  is unique. By  Claim~\ref{c.all-translation},  the third and forth items of Lemma~\ref{l.limit-maps}, one has 
\begin{itemize}
	\item $\Id_L\in\cL^+(L)$;
	\item for any $F,G\in\cL^+(L)$, $F\circ G\in\cL^+(L)$.
\end{itemize}
  To prove that $\cL^+(L)$ is a group, one needs to check that for any $F\in\cL^+(L)$, there exists $G\in\cL^+(L)$ such that $F\circ G=G\circ F=\Id_L.$ Let $F\in\cL^+(L)$ such that $F(x)=y$ for some $y\in L$. As $L\subset \cN$, there exists $G\in\cL^+(L)$ such that $G(y)=x$. Then the limit center map $G\circ F$ has a fixed point, by the forth item of Lemma~\ref{l.limit-maps}, $G\circ F=\Id_L$. By the fifth item of Lemma~\ref{l.limit-maps}, $F$ and $G$ are homeomorphisms on $L$, therefore $F\circ G=\Id_L.$ 
  
  To summarize, one obtains that the group $\cL^+(L)$ acts freely on $L$ and the action is faithful. By H\"older theorem (i.e. Theorem~\ref{thm.holder}),  the group $\cL^+(L)$ is isomorphic to the group of translations (resp. rotations) on $\RR$ (resp. $S^1$) if $L$ is homeomorphic to $\RR$ (resp. $S^1$). 
  As each orientation reversing limit center map from $L$ to $L$   is an involution,    $\cL(L)$ is a group; moreover,  $\cL(L)$ either  coincides  with $\cL^+(L)$ or is the group generated by $\cL^+(L)$ and $-\Id_L$.
\end{proof}

Next remark explains why these properties are key points for the proof of Theorem~\ref{thm.main-metric}:

\begin{Remark} The Euclidean metric on $\RR$ (resp. on $\RR/\ZZ$) is invariant under the action of  the group generated by the  translations and $-\Id_\RR$  (resp. by the rotations and $-\Id_{S^1}$)  and any invariant metric by the set of translations (resp.   rotations) is obtained by multiplying the Euclidean metric by a scalar. 
\end{Remark}

Lemma~\ref{l.exists-limit-maps} gives that for any $x,y\in M$,  if $y\in\omega(x)$, then there exists a limit center map from $L_x$ to $L_y;$ this allows us to build  the connections between  the limit center maps on different center leaves.
\begin{lemma}~\label{l.limit-map-between-different-center}
		Let $f$ be a $C^1$ partially hyperbolic diffeomorphism. Assume that $f$ is transitive and  has $1$-dimensional topologically neutral center.

	Then for any  two center leaves $L_1$ $L_2$ each of which contains a point in $\cN$, one has 
	 \begin{itemize}
	\item each limit center map from $L_1$ to $L_2$ is a homeomorphism; 
	\item for any limit center maps $F, G\in\cL(L_1,L_2)$,  there are $F_1\in \cL(L_1)$ and $F_2\in \cL(L_2)$ so that
	$$G=F\circ F_1= F_2\circ F.$$
\end{itemize}

\end{lemma}
\proof
By Proposition~\ref{p.limit-group}  and the assumption, $L_1\cup L_2$ is contained in $\cN.$ 

Let $H\in\cL(L_1,L_2)$ be a limit center map. By Lemma~\ref{l.exists-limit-maps} and the fact that $L_2\subset \cN$, for a point  $x\in L_1$, there exists a limit center map $\Phi: L_2\to L_1$ with $\Phi(H(x))=x$. By the third item of Lemma~\ref{l.limit-maps}, $\Phi\circ H$ is a limit center map from $L_1$ to $L_1$ which is a homeomorphism due to the forth item of Lemma~\ref{l.limit-maps}. Therefore $H$ is injective. As $H$ is surjective, one has that $H$ is a homeomorphism with $H^{-1}=\Phi.$

As the center bundle is one dimensional, one can give an orientation to $L_1$ and $L_2$ respectively such that $F$ preserves the orientation. 
As $F$ and $G$ are surjective, there exist $x_1,x_2\in L_1$ such that $F(x_1)=G(x_2)$. By Proposition~\ref{p.limit-group}, there exists a limit center map $F_1\in\cL^+(L_1)$ such that $F_1(x_2)=x_1.$ Therefore $F\circ F_1(x_2)=G(x_2)$. Let $H: L_2\to L_1$ be a limit center map with $H\circ G(x_2)=x_2$.  If $G$ also preserves the orientation,    then  $H\circ (F\circ F_1)=H\circ G$ since they have a common fixed point and simultaneously preserve or reverse the orientation , which implies that $F\circ F_1=G.$
If $G$ reverses the orientation, then one of the maps $H\circ G$ and $H\circ F\circ F_1$ reverses the orientation, thus   $\cL(L_1)$ contains an involution. By Proposition~\ref{p.limit-group}, there exists an involution $R\in\cL(L_1)$ having $x_2$ as a fixed point, then  the map $F\circ (F_1\circ R)$   reverses the orientation. An analogous argument as above gives that $F\circ (F_1\circ R)=G. $

Similarly, one can show that there exists $F_2\in\cL(L_2)$ with $F_2\circ F=G.$
\endproof 


The first item of Lemma~\ref{l.limit-map-between-different-center}  allows us to consider the image $F_*(\ell)$ of a
$\cL(L)$-invariant  metric $\ell$ on a center leaf $L\subset \cN$ by a limit center map $F\in \cL(L,L_1)$ for $L_1\subset \cN$, as $F$ is a homeomorphism. The second item gives that the metric $F_*(\ell)$ on $L_1$ is independent of the choice of $F$ and is  $\cL(L_1)$-invariant.

\begin{corollary}
	Consider a center leaf $L$ containing a point in $\cN$ (equivalently, included in $\cN$), and fix a  $\cL(L)$-invariant metric $\ell_L$ on $L$.  

For any center leaf $L_1\subset \cN$ and any two limit center maps $F_1,F_2\in \cL(L,L_1)$,  the image metrics $(F_1)_*(\ell_L)$ and $(F_2)_*(\ell_L)$ are equal:

$$(F_1)_*(\ell_L) = (F_2)_*(\ell_L).$$
Let us denote $\ell_{L_1}$ this metrics.  Then $\ell_{L_1}$ is invariant under the action of $\cL(L_1)$. 
\end{corollary}

\begin{Remark}~\label{r.f-invriance-of-metric}
  Let $L\subset \cN$ be a center leaf  and $\ell_L$ be a $\cL(L)$-invariant metric on $L$. Then $f(L)\subset \cN$. Furthermore, for any $F\in \cL(L,f(L))$, notice that $f^{-1}\circ F\in\cL(L)$, thus
$$F_*(\ell_L)=f_*(\ell_L).$$
\end{Remark}

To summarize, we get the  next proposition: 

\begin{proposition}\label{p.family-metrics}
		Let $f$ be a $C^1$ partially hyperbolic diffeomorphism. Assume that $f$ is transitive and  has $1$-dimensional topologically neutral center.

	Then there is a family $\{\ell_L\}_{ L \mbox{ center leaf in } \cN}$ of metrics in the center leaves contained in $\cN$ so that 
\begin{itemize}
 \item for any center leaves $L_1,L_2$ contained in $\cN$ and any $F\in\cL(L_1,L_2)$ one has 
 $$F_*(\ell_{L_1})=\ell_{L_2}$$
 \item for any center leaf $L\subset \cN$ one has 
 $$f_*(\ell_L)=\ell_{f(L)}.$$
\end{itemize}
Furthermore, if  $\{\tilde\ell_L\}_{L \mbox{ center leaf  in } \cN}$ is another family of metric satisfying the properties above, then there is $\lambda>0$ so that for any $L\subset \cN$ one has 
$$\tilde\ell_L=\lambda\cdot\ell_L.$$
 
\end{proposition}

Thus to prove  Theorem~\ref{thm.main-metric},  it remains to show that the family of metrics 
$\{\ell_L\}_{ L \mbox{ center leaf  in } \cN}$ extends in a continuous way  as a center metric on all $M$. The main tool for proving that is to check that the family  $\{\ell_L\}$ is invariant by the holonomies of the strong stable and strong unstable foliations, which is the aim of next section.

\subsection{Holonomy invariance and continuity: Ending the proof of Theorem~\ref{thm.main-metric}}

In this section, we  keep the notations from Section~\ref{s.limit-center-map-for-transitive-diffeo}. 
The following lemma tells us that the strong stable holonomy is well defined restricted to $\cN$.
\begin{lemma}\label{l.injective-holonomy} 
		Let $f$ be a $C^1$ partially hyperbolic diffeomorphism. Assume that $f$ is transitive and  has $1$-dimensional topologically neutral center. Let $L_1$, $L_2$ be two center leaves contained in $\cN$ and  in the same center-stable leaf $L^{cs}$.  
		
		Then the holonomy of the strong stable foliation induces a homeomorphism from $L_1$ to $L_2$. 
\end{lemma}
\begin{proof} Recall that the center stable foliation has the completeness property (due to Proposition~\ref{p.complete}).  Thus our assumption says that $L_2$ is contained in the union of the strong stable leaves through $L_1$  which coincides with $L^{cs}$ and conversely.  According to Corollary~\ref{c.strong-intersection},  each strong stable leaf cuts $L_1$ in at most $1$ point, and the same for $L_2$.  Then each strong stable leaf in $L^{cs}$ cuts $L_1$ and $L_2$ in exactly $1$ point respectively  inducing a $1$ to $1$ correspondence, and proving the lemma. 
\end{proof}

\begin{lemma}~\label{l.holonomy-invariance}
	Let $L_1,L_2$ be two center leaves contained in $\cN$ and  in a same center-stable leaf $L^{cs}$.  Let $H^{ss}\colon L_2\to L_1$ be the holonomy of the strong stable foliation given by Lemma~\ref{l.injective-holonomy},  and  $\{\ell_L\}_{ L \mbox{ center leaf in } \cN}$ be a family of metrics in the center leaves, given by Proposition~\ref{p.family-metrics}.  

Then $\ell_{L_1}= (H^{ss})_*(\ell_{L_2})$. 
\end{lemma}
\begin{proof} 
	Let us fix a sequence $n_i\to +\infty$ so that the restriction $f^{n_i}|_{L_2}$ converges to a limit center map $F\in\cL^s(L_2,L_1)$. 

According to Proposition~\ref{p.family-metrics} one has 
$$F_*(\ell_{L_2})= \ell_{L_1}.$$

On the other hand, as we are iterating positively, points in the same strong stable leaf have the same limit, therefore the restriction $f^{n_i}|_{L_1}$ converges to 
$F\circ (H^{ss})^{-1}= \tilde F\in\cL^s(L_1,L_1)$. Thus

$$(F\circ (H^{ss})^{-1})_*(\ell_{L_1})= \tilde F_*(\ell_{L_1})= \ell_{L_1}.$$

One deduces 
$$ F_*( (H^{ss})^{-1}_*(\ell_{L_1}))= F_*(\ell_{L_2})$$
that is 
$$\ell_{L_2}= (H^{ss})^{-1}_*(\ell_{L_1})$$
which concludes the proof. 
\end{proof}
 Remark~\ref{r.f-invriance-of-metric} gives the $f$-invariance of the center metric   defined on $\cN$. The next proposition gives a continuous family of metric on all the center leaves, therefore  ends the proof of Theorem~\ref{thm.main-metric}.
\begin{proposition} 
		Let $f$ be a $C^1$ partially hyperbolic diffeomorphism. Assume that $f$ is transitive and  has $1$-dimensional topologically neutral center.
	Let $\{\ell_L\}_{ L \mbox{ center leaf in } \cN}$ be a family of metrics in the center leaves, given by Proposition~\ref{p.family-metrics}. 
	
	Then this family of metrics in the center leaves in $\cN$ can be  extended in a unique way, by continuity, to all center leaves, defining a center metric on $M$. 
\end{proposition}
\begin{proof} 
	We denote $s=\dim(E^{s})$ and $u=\dim(E^{u})$. 
We consider a finite open  cover $\{V_i\} $ of $M$ given by   compact $C^0$-foliated boxes 
$\varphi_i\colon [-2,2]^s\times [-2,2]^u\times [-2,2]\to M$ so that : 
\begin{itemize}
	 \item $V_i=\varphi_i( (-1,1)^s\times (-1,1)^u\times (-1,1))$;
 \item each square $[-2,2]^s\times\{y\}\times [-2,2]$ is contained in a center stable leaf, 
 \item each square $\{x\}\times [-2,2]^u\times [-2,2]$ is contained in a center unstable stable leaf. 

 \item for any two points $(x_1,y_1), (x_2,y_2)\in [-1,1]^s\times [-1,1]^u$,
 as both $W^{ss}_{loc}((x_1,y_1)\times[-2,2])\cap W^{uu}_{loc}((x_2,y_2)\times[-2,2])$ and $W^{uu}_{loc}((x_1,y_1)\times[-2,2])\cap W^{ss}_{loc}((x_2,y_2)\times[-2,2])$ consist of a unique center path $L_1$ and $L_2$ respectively, then the local strong stable (resp. strong unstable ) holonomy map sends $(x_1,y_1)\times[-1,1]$ into $L_1$ (resp. $L_2$) and its image is sent by the local strong unstable (resp. strong stable) holonomy map into the interior of $(x_2,y_2)\times[-2,2]$.
\end{itemize}

Let us denote $U_i=\varphi_i([-2,2]^s\times [-2,2]^u\times [-2,2]).$ 
 For each $p\in U_i$ (resp. $V_i$), we  denote by $L_{p}|_{U_i}$ (resp. $L_{p}|_{V_i}$)   the connected component of $L_p\cap U_i$ (resp. $L_p\cap V_i$) containing the point $p$, where $L_p$ denotes the center leaf through $p$.  

We define a metric $\ell_i$ on center segments contained in $V_i$ as follows. 
As $\cN$ is a dense subset of $M$, for each point $p\in V_i$, there exists a sequence of points $\{q_n\}_{n\in\NN}\subset \cN$ with $q_n$ tends to $p$.  For $n_1$ and $n_2$ large, the intersection $W^s_{loc}(L_{q_{n-1}}|_{U_i})\cap W^u_{loc}(L_{q_{n_2}}|_{U_i})$ is non-empty and is  contained in $\cN$. As the center metric is invariant under strong stable and unstable holonomies, by the uniform continuity of the local strong stable and unstable holonomies in $U_i$, one deduces that the center metric on $L_{q_n}|_{U_i}$ uniquely induces a center metric on the $L_p|_{V_i}$, hence one gets a  metric on each center plaque in $V_i$,  moreover the uniqueness gives the continuity of the center metric in $V_i$.   Notice  that the center metric on each center path is independent of the choice of $V_i$ which allows us to define the center metric on the whole center leaf.  
Since the center metric on $\cN$ is invariant under the dynamics $f$ and invariant under the strong stable and unstable holonomies, by the continuity of the center metric and the strong stable and unstable holonomies, the center metric is invariant everywhere under the dynamics and the strong stable and unstable holonomies.

\end{proof}

%% file: Existence-periodic-compact.tex
\section{Existence of periodic compact center leaves }

In this section, we first work in any dimension and  show that for partially hyperbolic diffeomorphisms with $1$-dimensional topologically  neutral center, if there exist  compact center leaves, then there exist  periodic compact center leaves. Then we give some consequences in  dimension three.

The following general result is needed in this part.
\begin{lemma}~\label{l.intersects into compact center leaf}
Let $f$ be a dynamically coherent partially hyperbolic diffeomorphism with one-dimensional center and $L$ be a compact center leaf. For  $\e>0$ small, there exists $\delta>0$ such that for any compact center leaf $L^\prime$  in the $\delta$-tubular neighborhood of $L$,
	 the intersection $W^u_\e(L)\cap W^s_\e(L^\prime)$ consists of finitely many  compact center leaves.  
\end{lemma}
\proof
Let $\e_0>0$ be  small  enough such that one can defined a $\e_0$-tubular-neighborhood $V$ of $L$ together with a $C^1$ projection $\pi: V\to L$ such that each fiber $\pi^{-1}(x)$ (for $x\in L$) is transverse to the center foliation.

For $\e\in(0,\e_0)$, by the uniform transversality between  $E^{s}\oplus E^c$ and $E^u$,  there exists $\delta\in(0,\e/4)$  such that for any two points $x,y \in M$ with  $\ud(x,y)<\delta$, one has  
\begin{itemize}
	\item the intersection $\cF_\e^{uu}(x)\cap\cF^{cs}_\e(y)$ consists of exactly one point;
	\item $\cF_{\e/2}^{uu}(x)\cap\cF^{cs}_{\e/2}(y)=\cF_\e^{uu}(x)\cap\cF^{cs}_\e(y)$.
\end{itemize} 
 Let $L^\prime$ be a compact center leaf in the $\delta$-tubular neighborhood of $L$.
 Then for any $x\in L$, by the choice of $\delta$, one has that $\cF^{uu}_\e(x)\cap W^s_\e(L^\prime)$ consists of finitely many points and is $\e/2$ away from the boundaries  of $\cF^{uu}_\e(x)$ and $W^s_\e(L^\prime)$.  This gives that 
 $W^u_\e(L)\cap W^s_\e(L^\prime)$  consists of finitely many compact center leaves.
\endproof

In the following, we consider the case that there exists a compact center leaf $\gamma$ for a partially hyperbolic diffeomorphism with topologically neutral center. We will show that one can always find a compact and periodic center leaf. 
The proof uses the notion of bad sets for a compact lamination introduced  in ~\cite{E} and a  Bowen-type shadowing lemma given in appendix (see also ~\cite{BB, Ca}).
\begin{proposition}~\label{p.existence-of-peiodic-compact-center-leaf}
Let $f$ be a partially hyperbolic diffeomorphism. Assume that $f$ has  $1$-dimensional topologically neutral center and  admits a compact center leaf  $\gamma$.
Then $f$ has a compact periodic   center leaf. 

Moreover, if $\gamma$ is not periodic,   then there exists a compact periodic center leaf whose center stable manifold contains another different compact center leaf.
\end{proposition}
\proof
If $\gamma$ is periodic, we are done.

Now, we assume that $\gamma$ is non-periodic. Let $x\in\gamma$, then we consider the $\omega$-limit set $\omega(x)$ of $x$.
By  topologically neutral property,  there exists a compact $f$-invariant set $\La$ saturated by compact center leaves whose length are uniformly bounded. If $\La$ contains a compact periodic center leaf $L$, then for an arbitrarily small tubular neighborhood of $L$, by topologically neutral property, there exists $n\in\NN$ such that $f^n(\gamma)$ is entirely contained in the tubular neighborhood of $L$, thus, one can apply Lemma~\ref{l.intersects into compact center leaf} to conclude.

Now, we only need to deal with the case that $\La$ does not contain periodic center leaves.
We define a function $\ell:\La\mapsto\RR^+$ by associating $x\in\Lambda$ to the length of the center leaf through $x$.
By   continuity of the center foliation, the function $\ell$ varies lower semi-continuously. Now we define the bads set for $\ell$.
Let us denote $\La_0=\La$. For $i\in\NN$, one defines the $(i+1)$-th bad set by $$\La_{i+1}=\big\{x\in\La_i: \textrm{$\ell|_{\La_i}$ is not continuous at $x$}\big\}.$$ 
The $f$-invariance of the center foliation implies that $\La_i$ is  $f$-invariant.  Notice that $\ell|_{\La_i}$ is   continuous at $x\in\La_i$ if and only if the center holonomy group restricted to $\La_i$ for $\cF^c(x)$ is  trivial, hence the continuous points of $\ell_{\La_i}$ form  an open set which implies that $\La_{i+1}$ is compact.
Since the length of center leaves in $\La$ are uniformly bounded from above, there exists $i_0\in\La$ such that $\ell|_{\La_{i_0}}$ is a continuous map.  By Proposition~\ref{p.shadow-periodic}, arbitrarily close to $\La_{i_0}$, there exists a compact and periodic center leaf  $L$  
 whose stable manifold  contains another compact center leaf. 
\endproof

As an application, we obtain the following consequence on $3$-manifolds.
\begin{proposition}~\label{p.all-compact}
	Let $f$ be a transitive partially hyperbolic diffeomorphism on a closed 3-manifold $M$.
	Assume that 
	\begin{itemize}
		\item $f$ has one dimensional topologically neutral center;
		\item  there exist two different compact center leaves which are in the same center stable leaf.
	\end{itemize} 
Then up to finite lifts and iterates, $f$ is $C^0$-conjugate to a skew-product.
\end{proposition}
\proof
Let $\gamma$ and $\gamma^\prime$ be two compact center leaves of $f$ which are in the same center stable leaf. 
By Proposition~\ref{p.existence-of-peiodic-compact-center-leaf}, without loss of generality, one can assume that $\gamma$ is a periodic center leaf.  Thanks to Proposition~\ref{p.lift},  one can  assume that $f(\gamma)=\gamma$ for simplicity.

The compact leaves $\gamma^\prime$ and $\gamma$ bounds a region $C$ in $W^s(\gamma)$ which is an annulus or a M\"obius band. By Poincar\'e-Bendixson theorem, for each point $x\in C$, either $\cF^c(x)$ is compact or $\overline{\cF^c(x)}$ consists of $\cF^c(x)$ and two   compact center leaves in $C$. Since $f$ is transitive, there exist a point $x_0\in C$ and $n\in\NN$ such that   $f^{-n}(x_0)$ is in $W_{loc}^u(\gamma)$. One again,  by Poincar\'e-Bendixson, there exists a compact center leaf in $\overline{\cF^c(f^{-n}(x_0))}\cap W^u_{loc}(\gamma)$.  Since $\overline{\cF^c(f^{-n}(x_0))}\subset f^{-n}(C)\subset W^s(\gamma)$,  
 the intersection of  stable manifold and unstable manifolds of  $\gamma$    contains an entire compact center leaf. By the first item in Theorem~\ref{thm.bonatti-wilkinson},  modulo finite lifts and iterates, $f$ is $C^0$-conjugate to a  skew-product. 
\endproof
As a corollary of Propositions~\ref{p.existence-of-peiodic-compact-center-leaf} and~\ref{p.all-compact}, one has the following consequence.
\begin{corollary}
	Let $f$ be a   partially hyperbolic diffeomorphism   on a closed $3$-manifold.
	Assume that 
	\begin{itemize}
		\item $f$ is transitive and has $1$-dimensional topologically neutral center;
		\item $f$ simultaneously  has   compact and non-compact  center leaves, 
	\end{itemize} 
Then all the compact center leaves are periodic under $f$.
\end{corollary}

%% file: First-intersection.tex
\section{Classification of transitive partially hyperbolic diffeomorphisms with neutral center: Proof of Theorem~\ref{thm.main}}
In this section, we first recall the notion of $N$-th intersection of a hyperbolic saddle   for surface diffeomorphisms (introduced in~\cite{BL}) and some properties of $N$-th intersection sets. 
Then we extend this notion to partially hyperbolic setting for a compact periodic center leaf provided that the system is transitive and has topologically neutral center. 
At last, we give the proof of Theorem~\ref{thm.main}.
\subsection{$N$-th intersection of a hyperbolic saddle}
Now, we   introduce the notion of $N$-th intersection for a hyperbolic saddle of a surface diffeomorphism.

Let $f$ be a $C^1$-diffeomorphism on a surface $S$ and $p$ be a hyperbolic saddle. Assume that  the stable and unstable manifolds of $p$ are homeomorphic to $\mathbb{R}$, then for each $x\in W^s(p)$, we denote by $I^s_x$ the compact segment in $W^s(p)$ bounded by $p$ and $x$. Analogously, one   defines $I_x^u$ for $x\in W^u(p)$.

\begin{definition}
	Let $f$ be a $C^1$-diffeomorphism on a surface $S$ and  $p$ be a hyperbolic saddle. Assume that there is no homoclinic tangency between  the stable and unstable manifolds of $p$. A point $x\in W^s(p)\cap W^u(p)\setminus\{p\}$ is called the \emph{$N$-th intersection of the invariant manifolds of $p$}, if 
	$$\#\big(I_x^s\cap I_x^u\setminus\{p\}\big)=N.$$
We define a function $x\in W^s(p)\cap W^u(p)\setminus\{p\}\mapsto n(x)\in\mathbb{N}$ provided that $x$ is the $n(x)$-th intersection of the invariant manifolds of  $p$. See Figure 1.
\end{definition}
\begin{Remark}
	\begin{itemize}
		\item  Notice that the invariant manifolds of $p$ under $f$ coincide  with the invariant manifolds of $f^k$ for any $k\in \NN^+$, thus the $N$-th intersections of invariant manifolds under $f$ coincide with the ones under  $f^k$ for any $k\in\NN^+$;
		\item  For any $x\in W^s(p)\cap W^u(p)\setminus\{p\}$, one has $f((p,x))^s=(p,f(x))^s$ and $f((p,x)^u)=(p,f(x))^u$, which implies that $n(x)=n(f(x))$.
	\end{itemize}

\end{Remark}

  \begin{figure}[h]
	\begin{center}
		\def\svgwidth{0.5\columnwidth}
		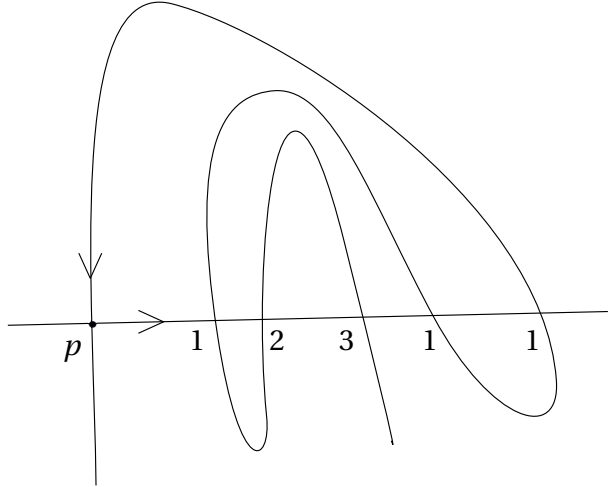
		\caption{Intersection for hyperbolic fixed points}
	\end{center}
\end{figure}

 In the following, we will show that for each $N\in\NN$ there are finitely many homoclinic orbits which are $j$-th intersection for $j\leq N.$
 \begin{proposition}~\label{p.finite-intersection}
 	Let $p$ be a hyperbolic saddle of a surface diffeomorphism $f$, and assume that $p$ has no homoclinic tangencies. For any $N\in\NN$, one has  
 	$$\#\big\{\Orb(x): x\in W^s(p)\cap W^u(p)\setminus\{p\}\textrm{ and }n(x)\leq N \big\}< +\infty.$$
 \end{proposition}
\proof
Denote by $J^{s,+}$ and $J^{s,-}$ the two separatrices of the stable manifold of $p$ and by $J^{u,+}$ and $J^{u,-}$ the two separatrices of the unstable manifold of $p$. 
Up to replacing $f$ by $f^2$, we may assume that $f$ preserves  these separatrices.
One only needs to prove the proposition for the intersections between   $J^{s,+}$ and $J^{u,+}$, and the rest case would follow analogously.

Let $x\in J^{s,+}(p)\cap J^{u,+}(p)$ be a homoclinic intersection of $p$ such that 
$$n(x)=\sup\big\{n(y): y\in J^{s,+} \cap J^{u,+}\setminus\{p\}\textrm{ and }n(y)\leq N\big\} .$$
Since for any  $z\in  J^{s,+}\cap J^{u,+}\setminus\{p\}$  the number $n(z)$ is invariant under $f$, without loss of generality, one can assume that $z\in I_x^s\setminus I^s_{f(x)}$. If $z\notin I_{f(x)}^u$, then $f(x)\in I_z^s\cap I_z^u$ which implies that  $n(z)\geq n(x)+1$. Therefore, for each homoclinic point $z\in  J^{s,+}\cap J^{u,+}\setminus\{p\}$ with $n(z)\leq n(x)$, up to finite iterates, one has 
$z\in \big(I_x^s\setminus I^s_{f(x)}\big)\cap I_{f(x)}^u$. Since there are no homoclinic tangencies for $p$, one has 
$$\#\big( \overline{I_x^s\setminus I^s_{f(x)}}\cap I_{f(x)}^u\big)<\infty,$$
ending the proof of Proposition~\ref{p.finite-intersection}.
\endproof

\subsection{$N$-th intersection for a periodic compact center leaf}

The idea is to `modulo the center foliation', and we `come to' the surface case and we define the $N$-th intersection for a periodic compact center leaf.  The difficulty comes from checking that the notion is well defined along the center leaves and is overcame by the center flows given by Theorem~\ref{thm.main-flow}.

Before defining the intersection number for a compact periodic center leaf, we need some preparations.
\begin{lemma}~\label{l.unique-intersection-for-periodic-center-leaf}
Let  $f$ be a partially hyperbolic diffeomorphism on a closed 3-manifold $M$. Assume that $f$ has  one dimensional topologically neutral center and  has a periodic compact center leaf $\gamma$.
Then one has 
\begin{itemize}
		\item  $\cF^{ss}(x)\cap\gamma=\cF^{uu}(x)\cap\gamma=\{x\},$ for each $x\in\gamma$;
	\item if $f$ is transitive, then  the intersection of $W^s(\gamma)\cap W^u(\gamma)$ is dense in $W^s(\gamma)$ and $W^u(\gamma)$.
\end{itemize}
	
	\end{lemma}
\proof
Assume, on the contrary, that there exist $x,y\in \gamma$ with  $x\in\cF^{ss}(y)$, then by iterating  $x$ and $y$ forwardly,  one gets that for any $\e>0$, there exists a point $z_\e\in\gamma$ such that $\cF^{ss}_{\e}(z_\e)$ intersects $\gamma$ into at least two points, which contradicts to the transversality  in $E^{cs}$ between $E^s$ and $E^c$.

Let $k$ be the period of the center leaf $\gamma$ under $f$. By Proposition~\ref{p.lift}, $f^k$ is still transitive and $\gamma$ is $f^k$-invariant, then one concludes by transitivity.
\endproof 
Let us fix some notations before defining the $N$-th intersection.
Let $f$ be a partially hyperbolic diffeomorphism  a closed $3$-manifold 
  with the following properties:
\begin{itemize}
	\item $f$   has $1$-dimensional topologically neutral center.	 
	\item $f$ admits  a periodic compact center leaf $\gamma$.
	\item  the bundles $E^s,E^c,E^u$ are   orientable.
\end{itemize}
 For $x\in W^u(\gamma)\cap W^s(\gamma)\setminus\{\gamma \}$, let $x^s=\cF^{ss}(x)\cap\gamma$ and $x^u=\cF^{uu}(x)\cap\gamma$ (which is unique due to Lemma~\ref{l.unique-intersection-for-periodic-center-leaf}). We denote by $I_{x,\gamma}^{ss}$  the compact  strong stable   segments bounded by $x$ and $x^s$. Analogously, one can define  $I_{x,\gamma}^{uu}$ associated to $\gamma,x,x^u.$
When there is no confusion, we will drop the index $\gamma$ for simplicity.
By the completeness of the invariant foliations, the center leaf $\cF^c(x)$ intersects $I_x^{ss}$ and $I_x^{uu}$ into infinitely many points respectively. Let $x^s_1\in\cF^c(x)\cap I_x^{ss}$ be a point such that  the open strong stable segment  $(x^s_1,x)^{ss}$ is disjoint from $\cF^c(x)$, and let $x^u_1\in\cF^c(x)\cap I_x^{uu}$ be the point analogously defined for the strong unstable.  Then the center segment $(x^s_1,x)^c$, the strong stable segment $I^{ss}_x$  and $\gamma$ bound  a compact center stable   submanifold   and we denote it  as $I^{cs}(x)$; likewise, one   gets a compact center unstable submanifold denoted as  $I^{cu}(x)$. See Figure 2.

 \begin{figure}[h]
	\begin{center}
		\def\svgwidth{0.5\columnwidth}
		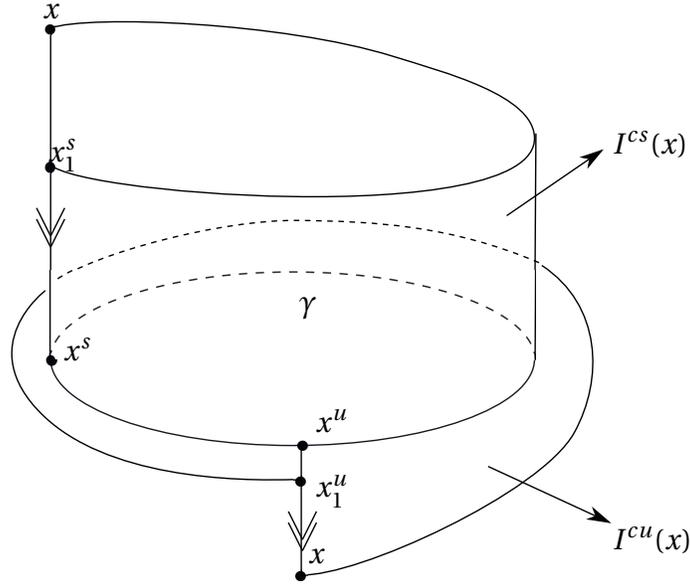
		\caption{$N$-th intersection for compact periodic center leaf}
	\end{center}
\end{figure}
To guarantee that the notion is well defined, we need to put more restrictions on the diffeomorphism than the case for a surface diffeomorphism. 
\begin{definition}
Consider a  partially hyperbolic diffeomorphism $f$ on a closed $3$-manifold with the following properties:
	\begin{itemize}
		\item $f$ is transitive and has $1$-dimensional topologically neutral center.	 
		\item $f$ admits  a periodic compact center leaf $\gamma$.
		\item  the bundles $E^s,E^c,E^u$ are   orientable.
	\end{itemize}
	We say that  $x\in W^s(\gamma)\cap W^u(\gamma)\setminus\{\gamma\}$ is \emph{the $N$-th intersection} of $\gamma$ if 
	$I^{cu}(x)\cap I^{cs}(x)\setminus \{\gamma\}$ has exactly $N$ connected components.
Then each $x\in W^{s}(\gamma)\cap W^u(\gamma)\backslash\gamma$ is   associated to  a number  $n(x)\in\NN$  if $x$ is the $n(x)$-th intersection.
\end{definition}
\begin{Remark}~\label{r.center-intersection-number}
 The invariant foliations of $f$ coincide with the corresponding  invariant foliations of $f^k$ for any $k\in\NN^+$, hence the $N$-th intersections   under $f$ coincide with the ones of $f^k$ for $k\in\NN^+$.  
\end{Remark}

\begin{lemma}~\label{l.invariance-number}
	The intersection number is well define, that is, 
	for each $x\in W^{s}(\gamma)\cap W^u(\gamma)\backslash\gamma$, one has 
	\begin{enumerate}
		\item $n(x)=n(f(x))$;
		\item $n(x)=n(y)$, for any $y\in\cF^c(x)$.
	\end{enumerate}
\end{lemma}
\proof
By    the invariance of the foliations, $I^{cu}(f(x))=f(I^{cu}(x))$ and $I^{cs}(f(x))=f(I^{cs}(x))$ which implies 
$n(x)=n(f(x))$.

By Remark~\ref{r.center-intersection-number} and Proposition~\ref{p.lift}, up to replacing $f$ by $f^2$, one can assume that $f$ preserves the orientation of $E^c$. 
By Theorem ~\ref{thm.main-flow}, there exists a center flow $\{\varphi_t
\}_{t\in\RR}$ commuting with the strong stable and unstable holonomies, therefore the center flow preserves the strong stable and unstable foliations. For any point $y\in\cF^c(x)$, there exists $t_0\in\RR$ such that $\varphi_{t_0}(x)=y.$ By definition and the fact that the center flow preserves the strong stable and unstable foliations, $\varphi_{t_0}(I^{cs}(x))=I^{cs}(y)$ and $\varphi_{t_0}(I^{cu}(x))=I^{cu}(y)$,  which implies 
 $n(x)=n(y)$ since $\varphi_{t_0}$ is a homeomorphism.
\endproof 

\begin{proposition}~\label{p.finite-center-leaf}
	Let $f$ be a   partially hyperbolic diffeomorphism  on a closed $3$-manifold with the following properties:
	\begin{itemize}
		\item $f$ is transitive and has $1$-dimensional topologically neutral center.	 
		\item $f$ admits  a periodic compact center leaf $\gamma$ and a non-compact center leaf.
		\item  the bundles $E^s,E^c,E^u$ are   orientable.
	\end{itemize}
	
Then	for any integer $N\in\NN$, there are finitely many center leaves where $n(\cdot)$ is bounded by $N$; in formula 
$$\#	\big\{\cF^c(x): \cF^c(x)\subset W^s(\gamma)\cap W^u(\gamma)\setminus\{\gamma\} \textrm{ and  } n(x)\leq N\big\}<\infty.$$
	\end{proposition}
 \begin{Remark}
 Here we prove the finiteness of center leaves where $n(\cdot)$ is bounded, whereas  Proposition~\ref{p.finite-intersection} gives the finiteness of   orbits with $n(\cdot)$ bounded.
 \end{Remark}

\proof

By Proposition~\ref{p.lift}, up to replacing $f$ by $f^2$, one can assume that $f$ preserves the orientation of the bundles $E^s,E^c, E^u.$ As  $E^s,E^c, E^u$ are orientable,   $\gamma$ separates its stable and unstable manifolds into two connected components respectively. Thus, we only need to work on one connected component $W^{s,+}(\gamma)$ of $W^{s}(\gamma)\setminus\gamma$ and one connected component $W^{u,+}(\gamma)$ of $W^{u}(\gamma)\setminus\gamma$,  and the other cases would follow analogously.

Fix $N\in\NN$ and   $x\in W^{u,+}(\gamma)\cap W^{s,+}(\gamma)$  such that $$n(x)=\sup\big\{n(y): y\in W^{u,+}(\gamma)\cap W^{s,+}(\gamma)\textrm{~and~}n(y)\leq N\big\}.$$

We keep the notations in the definitions of $I^{cu}(x)$ and $I^{cs}(x)$ (see Figure 2). Let $\{\varphi_{t}\}_{t\in\RR}$ be the center flow given by Theorem~\ref{thm.main-flow}. Without loss of generality, one can assume that $x_1^s$ is on the forward  orbit of $x$ under the center flow $\{\varphi_{t}\}_{t\in\RR}$. 
Let $x^s_{-1}$ and $x^s_{-2}$ be first and second   points that  the backward  orbit of $x$ under the center flow $\{\varphi_{t}\}_{t\in\RR}$ intersects with the strong stable manifold of $x$. Then $I^{cs}(x)\subset I^{cs}(x^s_{-1})\subset I^{cs}(x_{-2}^s)$. Analogously, one can define $x_{-1}^u$ in the strong unstable manifold of $x$, then $I^{cu}(x)\subset I^{cu}(x^u_{-1}).$ 

As $f$ has non-compact center leaves, by Proposition~\ref{p.all-compact}, $\gamma$ is the unique  compact center leaf in $W^s(\gamma)$ and in $W^u(\gamma)$. By Poincar\'e-Bendixson theorem, for each $\cF^c(x)\subset W^{s,+}(\gamma)\cap W^{u,+}(\gamma)\setminus\{\gamma\}$, one has   $ \gamma\subset \overline{\cF^c(x)}$.
For any $y\in W^{s,+}(\gamma)\cap W^{u,+}(\gamma)\setminus \big\{\gamma\cup\{x \}\big\}$,  by Lemma~\ref{l.invariance-number}, up to replacing $y$ by a point  in $\cF^c(y)$, one can assume that $y$ belongs to the strong unstable segment bounded by $x$ and $x^u_{-1}$, then $I^{cu}(x)\subset I^{cu}(y)\subset I^{cu}(x^u_{-1}) .$
\begin{claim}~\label{c.cs-included}
 If $n(y)\leq N$, then $I^{cs}(y)\subset I^{cs}(x^s_{-2})$.
\end{claim} 
\proof
Assume, on the contrary, that $I^{cs}(y)\nsubseteq  I^{cs}(x^s_{-2})$, then   $y\notin I^{cs}(x^s_{-1}) $ and  $I^{cs}(x^s_{-1})\subset I^{cs}(y)$, which implies that $\cF^c(y)$ is disjoint from $I^{cs}(x)\cap I^{cu}(x)$. Since $I^{cs}(x^s_{-1})\subset I^{cs}(y)$ 	and $I^{cu}(x)\subset I^{cu}(y)$, 
the cardinal of the connected components of  $(I^{cs}(y)\cap I^{cu}(y))$ is larger than the cardinal of the connected components of $I^{cs}(x)\cap I^{cu}(x)$, which  gets the contradiction.
\endproof
By Claim~\ref{c.cs-included}, for any $y\in W^{s,+}(\gamma)\cap W^{u,+}(\gamma)\setminus \big\{\gamma\cup\{x \}\big\}$ with $n(y)\leq N$, one has   $$\cF^c(y)\cap \big(I^{cs}(x^s_{-2})\cap I^{cu}(x^u_{-1})\big)\neq\emptyset.$$ By the compactness of $I^{cs}(x^s_{-2})$and $ I^{cu}(x^u_{-1})$, and the uniform transversality between $E^{cs}$ and $E^{cu}$, the set $I^{cs}(x^s_{-2})\cap I^{cu}(x^u_{-1})$ has finitely many  connected components, which implies 
$$\#	\big\{\cF^c(y): \cF^c(y)\subset W^s(\gamma)\cap W^u(\gamma)\setminus\{\gamma\} \textrm{ and  } n(y)\leq N\big\}<\infty.$$
\endproof

We conclude this section by  the following result. 
\begin{corollary}~\label{c.center-leaf-all-periodic}
Under the assumption of Proposition~\ref{p.finite-center-leaf}, 	all the center leaves in $W^s(\gamma)$ and $W^u(\gamma)$ are periodic under $f$.
\end{corollary}
\proof
We claim that each center leaf in $\cF^c(x)\subset W^s(\gamma)\cap W^u(\gamma)$ is periodic under $f$. By Proposition~\ref{p.finite-center-leaf}, one has  
$$\#	\big\{\cF^c(y): \cF^c(y)\subset W^s(\gamma)\cap W^u(\gamma) \textrm{ and  } n(y)\leq n(x)\big\}<\infty.$$
By Lemma~\ref{l.invariance-number}, for each $k\in\ZZ$, one has 
$$\cF^c(f^k(x))\subset	\big\{\cF^c(y): \cF^c(y)\subset W^s(\gamma)\cap W^u(\gamma) \textrm{ and  } n(y)\leq n(x)\big\},$$ 
which implies that $\cF^c(x)$ is periodic under $f$.

By Lemma~\ref{l.unique-intersection-for-periodic-center-leaf}, the intersection $W^s(\gamma)\cap W^u(\gamma)$ is a dense subset of $W^s(\gamma)$.  
As $W^s(\gamma)$ is a cylinder and $\gamma$ is periodic under $f$, in each connected component of $W^s(\gamma)\setminus\{\gamma\}$, the space of center leaves is identified with $S^1$ and $f$ induces a homeomorphism on it.  Therefore, the set of periodic points for  the induced maps on $S^1$  is dense in $S^1$, which implies that the induced maps on $S^1$ are periodic. 
\endproof
 
\subsection{Proof of Theorem~\ref{thm.main}}
Now, we are ready to give the proof of Theorem~\ref{thm.main}. The proof is carried out according to the topology of the center stable leaves. 

\proof[Proof of Theorem~\ref{thm.main}]
By Proposition~\ref{p.existence-of-compact-center-leaf}, $f$ has compact center leaves.

If there exists a  compact center leaf which is non-periodic under $f$, then  by the `moreover' part of Proposition~\ref{p.existence-of-peiodic-compact-center-leaf} there exists a compact periodic center leaf and Proposition~\ref{p.all-compact} gives us that $f$ is, up to finite lifts,  $C^0$-conjugate to a skew-product. Therefore, up to finite lifts, $f$ is conjugate to a skew-product and also $f$ preserves a volume on the center fibers ($S^1$),  thus, $f$ is conjugate to a skew-product of an Anosov diffeomorphism on $\mathbb{T}^2$ over the rotations on the circle.

It remains to prove the case where   all  the compact center leaves are periodic under $f$. By Proposition~\ref{p.lift}, up to finite iterates and lifts, $f$ satisfies the assumption of Proposition~\ref{p.finite-center-leaf}. By Corollary~\ref{c.center-leaf-all-periodic} and the second item in Theorem~\ref{thm.bonatti-wilkinson},   up to finite iterates and lifts, each center leaf is $f$-invariant.  Let $(\varphi_t)_{t\in\RR}$ be the center flow given by Theorem~\ref{thm.main-flow}. Let $x_0$ be a point whose orbit under $f$ is dense. 
As each center leaf is $f$ invariant and $f$ commutes with the center flow. 
Then there exists $t_0\in\RR\setminus\{0\}$ such that $f|_{\cF^c(x_0)}=\varphi_{t_0}|_{\cF^c(x_0)}$. Since the orbit of $x_0$ is dense and $f$ commutes with the center flow, one has $f=\varphi_{t_0}$.  In particular, this implies the center flow is transitive. Moreover, there exists $\lambda>0$ such that for any two points $x,y $ on the same strong stable  manifold of $f$, one has 
$$\limsup_{t\rightarrow+\infty}\frac{1}{t}\log\ud\big(\varphi_t(x),\varphi_t(y)\big)< -\lambda$$
An analogous statement for strong unstable also holds.
\endproof

%% file: drawing-1.eps_tex
\begingroup%
  \makeatletter%
  \providecommand\color[2][]{%
    \errmessage{(Inkscape) Color is used for the text in Inkscape, but the package 'color.sty' is not loaded}%
    \renewcommand\color[2][]{}%
  }%
  \providecommand\transparent[1]{%
    \errmessage{(Inkscape) Transparency is used (non-zero) for the text in Inkscape, but the package 'transparent.sty' is not loaded}%
    \renewcommand\transparent[1]{}%
  }%
  \providecommand\rotatebox[2]{#2}%
  \ifx\svgwidth\undefined%
    \setlength{\unitlength}{595.27559055bp}%
    \ifx\svgscale\undefined%
      \relax%
    \else%
      \setlength{\unitlength}{\unitlength * \real{\svgscale}}%
    \fi%
  \else%
    \setlength{\unitlength}{\svgwidth}%
  \fi%
  \global\let\svgwidth\undefined%
  \global\let\svgscale\undefined%
  \makeatother%
  \begin{picture}(1,0.8)%
    \put(0,0){\includegraphics[width=\unitlength]{drawing-1.eps}}%
    \put(0.09112702,0.2208235){\color[rgb]{0,0,0}\makebox(0,0)[lb]{\smash{$p$}}}%
    \put(0.30290464,0.22926403){\color[rgb]{0,0,0}\makebox(0,0)[lb]{\smash{1}}}%
    \put(0.4357,0.2293308214){\color[rgb]{0,0,0}\makebox(0,0)[lb]{\smash{2}}}%
    \put(0.5517533,0.22913180949){\color[rgb]{0,0,0}\makebox(0,0)[lb]{\smash{3}}}%
    \put(0.689246,0.22913308214){\color[rgb]{0,0,0}\makebox(0,0)[lb]{\smash{1}}}%
    \put(0.86,0.22914071844){\color[rgb]{0,0,0}\makebox(0,0)[lb]{\smash{1}}}%
  \end{picture}%
\endgroup%

%% file: drawing2.eps_tex
\begingroup%
  \makeatletter%
  \providecommand\color[2][]{%
    \errmessage{(Inkscape) Color is used for the text in Inkscape, but the package 'color.sty' is not loaded}%
    \renewcommand\color[2][]{}%
  }%
  \providecommand\transparent[1]{%
    \errmessage{(Inkscape) Transparency is used (non-zero) for the text in Inkscape, but the package 'transparent.sty' is not loaded}%
    \renewcommand\transparent[1]{}%
  }%
  \providecommand\rotatebox[2]{#2}%
  \ifx\svgwidth\undefined%
    \setlength{\unitlength}{595.27559055bp}%
    \ifx\svgscale\undefined%
      \relax%
    \else%
      \setlength{\unitlength}{\unitlength * \real{\svgscale}}%
    \fi%
  \else%
    \setlength{\unitlength}{\svgwidth}%
  \fi%
  \global\let\svgwidth\undefined%
  \global\let\svgscale\undefined%
  \makeatother%
  \begin{picture}(1,1)%
    \put(0,0){\includegraphics[width=\unitlength]{drawing2.eps}}%
    \put(0.05309494,0.9377585){\color[rgb]{0,0,0}\makebox(0,0)[lb]{\smash{$x$}}}%
    \put(0.49442335,0.03105527){\color[rgb]{0,0,0}\makebox(0,0)[lb]{\smash{$x$}}}%
    \put(0.0875,0.36928){\color[rgb]{0,0,0}\makebox(0,0)[lb]{\smash{$x^s$}}}%
    \put(0.50621164,0.141371){\color[rgb]{0,0,0}\makebox(0,0)[lb]{\smash{$x_1^u$}}}%
    \put(0.47987,0.45){\color[rgb]{0,0,0}\makebox(0,0)[lb]{\smash{$\gamma$}}}%
    \put(0.50601898,0.25){\color[rgb]{0,0,0}\makebox(0,0)[lb]{\smash{$x^u$}}}%
    \put(0.0623880312,0.700000666148){\color[rgb]{0,0,0}\makebox(0,0)[lb]{\smash{$x^s_1$}}}%
    \put(1,0.712955){\color[rgb]{0,0,0}\makebox(0,0)[lb]{\smash{$I^{cs}(x)$}}}%
    \put(1,0.05550656){\color[rgb]{0,0,0}\makebox(0,0)[lb]{\smash{$I^{cu}(x)$}}}%
  \end{picture}%
\endgroup%

%% file: Shadowing.tex
\appendix
\section{Periodic compact center leaves generated by a uniformly compact lamination}

In this section, we prove the existence of periodic compact  center leaves near a compact invariant set which is laminated by compact center leaves. The proof adopts a variation of Bowen's ~\cite{Bow} construction of shadowing lemma for hyperbolic sets which has been used in ~\cite{Ca}.

\begin{proposition}~\label{p.shadow-periodic}
	Let $f$ be a dynamically coherent partially hyperbolic diffeomorphism on $M$   and  $\Lambda$ be a compact invariant set. Assume that 
	\begin{itemize}
		\item every center leaf through  $x\in\Lambda$ is compact and contained in $\La$;
		\item the volume of the center leaves vary continuously restricted to $\Lambda$. 
	\end{itemize} 
Then for any $\e>0$,  there exists a compact and periodic center leaf $L_\e$ in the $\e$-neighborhood of $\Lambda$. Furthermore, if $L_\e\cap\Lambda=\emptyset$, the center stable leaf of $L_\e$ contains another compact  center leaf different from $L_\e$.
	
\end{proposition}
\proof
As each center leaf   in $\Lambda$ is compact,  one associates each center leaf $L\subset\Lambda$  to a tubular neighborhood $V_L$ of $L$ together with the  $C^1$-projection $\pi_L: V_L\rightarrow L$ such that for any $x\in L$, $\pi_L^{-1}(x)$ is a $C^1$ disc of co-dimension $\dim(L)$ which is transverse to the center foliation (see for instance~\cite[Chapter IV, Lemma 2]{CN}).  As the volume of the center leaves vary continuously in $\Lambda$, up to shrinking $V_L$, one can assume that
\begin{itemize}
	\item for any $x\in \Lambda\cap V_L$, the center leaf  $L_x$ is contained in $ V_L$;
	\item for each $y\in L$, the intersection $L_x\cap\pi_L^{-1}(y)$ is unique. 
\end{itemize}   
Then by the compactness of $\Lambda$, there exist compact center leaves  $L_1,\cdots,L_m$ in $\Lambda$ such that  their tubular neighborhoods $(V_{L_i},\pi_{L_i})$ chosen as above  form an open cover of $\Lambda$ (i.e. $\Lambda\subset \cup_{i=1}^m V_{L_i}$). For simplicity, we denote $V_i=V_{L_i}$ and $\pi_i=\pi_{L_i}$. By a standard argument, one gets $\delta_0>0$ such that for any center leaf $L\subset \Lambda$, there exists $1\leq i\leq m$ such that the $\delta_0$-tubular neighborhood of $L$ is in $V_i.$

Fix $\delta\in(0,\delta_0/2)$  and  define $\Lambda(\delta)$ as the set of points $x\in M$ with the following properties: 
\begin{itemize}
	\item center leaf $L_x$ is compact;
	\item  there exists a center leaf $L\subset \Lambda$  such that  $L_x$  is in  the closure of the $\delta$-tubular neighborhood of $L$;
	\item  $L_x$ intersects each fiber of  $\pi_i$ into a unique point, where $V_i$ contains the $\delta_0$-tubular neighborhood of $L$. 
\end{itemize}  
By definition, $\Lambda(\delta)$ is compact.
Notice that for any $\e\in(0,\delta_0/8)$ small enough one has that  for any two points $x,y\in M$,  
\begin{itemize}
	\item if $W^{ss}_{2\e}(x)\cap W_{2\e}^{cu}(y)\neq\emptyset$ (resp. $W^{uu}_{2\e}(x)\cap W_{2\e}^{cs}(y)\neq\emptyset$ ), then such intersection consists of a unique point;
	\item if  $\ud(x,y)>\delta$, then  $W_{2\e}^{cs}(y)\cap W^{cu}_{2\e}(x)=\emptyset$. 
\end{itemize} 
For $\e$, there exists  $\delta_1\in(0,\delta)$  such that for $x,y\in M$ with $\ud(x,y)<\delta_1$, one has 
\begin{itemize}
	\item $\cF^{ss}_{2\e}(x)\cap \cF^{cu}_{2\e}(y)=\cF^{ss}_{\e/2}(x)\cap \cF^{cu}_{\e/2}(y)$;
	\item $\cF^{ss}_{2\e}(x)\cap \cF^{cu}_{2\e}(y)$ consists of exactly one  point.
\end{itemize}

\begin{claim}~\label{c.unique-center-leaf}
Given  two compact center leaves $L_1,L_2\in\Lambda(\delta)$ satisfying that $L_1$ is contained in the $\delta_1$-tubular neighborhood of $L_2$, one has 
\begin{itemize}
	\item $W_{2\e}^s(L_1)\cap W_{2\e}^u(L_2)=W_{\e/2}^s(L_1)\cap W_{\e/2}^u(L_2)$;
	\item $W_{2\e}^s(L_1)\cap W_{2\e}^u(L_2)$ consists of  exactly one  compact center leaf $L$.
\end{itemize}  Moreover,  for $x\in L_1$ (resp. $x\in L_2$), $L$ intersects $W^{ss}_{\e/2}(x)$ (resp. $W^{uu}_{\e/2}(x)$)  into a unique point. 
\end{claim}
\proof
By the definition of $\Lambda(\delta),$ there exists  $1\leq i_0\leq m$ such that 
 $V_{i_0}$ contains $L_1$ and $L_2$.
 Furthermore, for any point $y\in L_{i_0}$, the transverse section $\pi_{i_0}^{-1}(y)$ cuts $L_1$ and $L_2$ into a unique point respectively and we denote them by $y_1,y_2$. As $L_1$ is contained in the $\delta_1$-tubular neighborhood of $L_2$, one has that $W_{2\e}^{ss}(y_1)\cap W_{2\e}^{cu}(y_2)$ (resp. $W_{2\e}^{cs}(y_1)\cap W^{uu}_{2\e}(y_2)$) consists of  a unique point which is $\e/2$ close to $y_1$ and to $y_2$.  By the choice of $\e$, the intersection $W^{ss}_{2\e}(y_1)\cap W^{cu}_{2\e}(L_2)$ (resp. $W^{uu}_{2\e}(y_2)\cap W^{cu}_{2\e}(L_1)$) consists of exactly one  point which is $\e/2$ close to $y_1$ (resp. $y_2$), which concludes the claim.
\endproof
 

Since $\Lambda$ is compact and $f$-invariant, there exists a recurrent point $x_0\in \Lambda$.
 Due to the continuity of the volume of center leaves in $\Lambda$,  and the uniform contraction and expansion along $E^{s}$ and $E^{u}$ respectively, there exists    $k\in\NN$    such that 
\begin{itemize}
	\item $f^k(L_{x_0})$ is in the $\delta_1/4$-neighborhood of $L_{x_0}$;
	\item  $f^k(W^{ss}_\e(y))\subset W^{ss}_{\delta_1/4}(f^k(y)) \textrm{~ and ~} f^{-k}(W^{uu}_\e(y))\subset W^{uu}_{\delta_1/4}(f^{-k}(y)), \textrm{~ for any $y\in M$};$
	\item $\max\big\{\sup_{x\in M}\|Df^k|_{E^{s}(x)}\|,\sup_{x\in M}\|Df^{-k}|_{E^{u}(x)}\|\big\}<1/4.$
\end{itemize}

 By Claim~\ref{c.unique-center-leaf}, $W_{\e/2}^u(f^k(L_{x_0}))\cap W_{\e/2}^s(L_{x_0})$ consists of exactly one compact center leaf $\hat{L}_1\subset \Lambda(\delta)$. Assume that we already get  compact center leaves $\{\hat{L}_j\}_{j\leq i-1}\subset\Lambda(\delta)$   such that for $j\leq i-1$, one has
 \begin{itemize}
 	\item   $\hat{L}_{j}\subset W^u_{\e/2} (f^k(\hat{L}_{j-1}))\cap   W^s_{\e/2}(L_{x_0})$;
 	\item  $\hat L_{j}$ intersects $W^{uu}_{\e/2}(z)$ into a unique point for each $z\in f^k(\hat{L}_{j-1})$;
 	\item   $\hat{L}_j$ intersects $W^{ss}_{\e/2}(z)$ into a unique point for  each $z\in L_x$.
 \end{itemize}
 By the choice of $k$, one has $f^k(\hat{L}_{i-1})\subset W^{s}_{\delta_1/4}(f^k(L))$, then once again by Claim~\ref{c.unique-center-leaf}, the intersection $W^u_{\e/2}(f^k(\hat{L}_{i-1}))\cap W^s_{\e/2}(L) $ consists of exactly  compact center leaf $\hat{L}_i$ which by definition is contained in $\Lambda(\delta)$.

 Let $L_i=f^{-ik}(\hat{L}_{2i})$ for each $i\in\NN$. 
By construction, one has 
\begin{itemize}
	\item $L_i\subset \Lambda(\delta)$;
	\item $L_i$ is contained in the $2\e$-tubular neighborhood of $L$;
	\item $W^u_\e(L_i)$ (resp. $W^s_\e(L_i)$) intersects $W^s_\e(L_{x_0})$ (resp. $W^u_\e(f^k(L_{x_0}))$) into a unique compact center leaf;
	\item   $\{f^j(L_i)\}_{j=-ik}^{ik}$ is in the $2\e$-tubular neighborhood of $\{f^{j}(L_{x_0})\}_{j\in [-ik,ik] (mod~k)}$.
\end{itemize}  
 Let $L$ be an accumulation of $\{L_i\}_{i\in\NN}$, then $L$ is a compact center leaf contained in the $\e$-tubular-neighborhood of $L_{x_0}$. Furthermore, $f^j(L)\subset\Lambda(\delta)$ is contained in the $2\e$-tubular-neighborhood of $f^{j-[j/k]k}(L)$, thus $f^k(L)$ has the same property, which implies that the orbit of $f^k(L)$ follows the orbits of $L$ in the distance of $2\e$. Applying item (c) of Theorem 6.1 in ~\cite{HPS} to $\cap_{n\in\ZZ} f^n(\Lambda(\delta))$, one has that $f^k(L)\subset W^s_{2\e}(L)\cap W^{u}_{2\e}(L)$. By Claim~\ref{c.unique-center-leaf}, one has that $f^k(L)=L$.
\endproof

%% file: Biblio.tex
\bibliographystyle{plain}